\documentclass[a4paper, 11pt]{article}

\usepackage{mathrsfs,amssymb,amsmath, amsthm,tikz, mathdots, amsfonts, graphicx, subfigure, tikz,xcolor, float, pgf, mathrsfs, verbatim}

\usetikzlibrary{arrows}
\usepackage[top=1.31in,bottom=1.29in,left=1.1in,right=0.9in]{geometry}
\usepackage[numbers]{natbib}
\usepackage[all]{xy}
\usepackage{epic}

\definecolor{darkgreen}{rgb}{0,.6,0}
\definecolor{lightg}{gray}{0.8}

\setcitestyle{open={},close={}}

\numberwithin{equation}{section}\theoremstyle{definition}
\swapnumbers

 \newtheorem{Theorem}[equation]{Theorem}
 \newtheorem{Prop}[equation]{Proposition}
 \newtheorem{Lemma}[equation]{Lemma}
 \newtheorem{Cor}[equation]{Corollary}
 \newtheorem{Notation}[equation]{Notation}
 \newtheorem{Defn}[equation]{Definition}%[section]
 \newtheorem{Example}[equation]{Example}%[section]
 \newtheorem{Remark}[equation]{Remark}

 \makeatletter
\def\enumerate{\begingroup\ifnum\@enumdepth>3\@toodeep\else
      \advance\@enumdepth\@ne
      \edef\@enumctr{enum\romannumeral\the\@enumdepth}%
      \topsep\z@\parskip\z@
      \list{\csname label\@enumctr\endcsname}
        {\@nmbrlisttrue\let\@listctr\@enumctr
         \parsep\z@\itemsep\z@\topsep\z@
         \setcounter{\@enumctr}{0}
         \def \fMakelabel##1{\hss\llap{\rm ##1}}
       }\fi}

 \makeatother

 \def\tril{\raisebox{-0.15em}{\tikz{\draw[black, fill=lightg, thick](0,0)--(0.5em,0)--(0.5em,-0.5em)--cycle;  \draw [densely dotted, thick] (0.5em,0)--(1em,0)--(1em,-1em)--(0.5em,-0.5em);}}}

 \def\trir{\raisebox{-0.15em}{\tikz{\draw[densely dotted, thick](0.5em,0)-- (0,0)--(1em,-1em)--(1em,0em); 
 \draw [black, fill=lightg,thick] (0.5em,0)--(1em,0)--(0.5em,-0.5em);}}}
 
 \def\pKL{\raisebox{-0.15em}{\tikz[scale=0.9]{\draw(0,0)--(1em,0)--(0em,-1em)--cycle;
\draw(0em,-0.2em)--(0.8em,-0.2em);
\draw(0em,-0.4em)--(0.6em,-0.4em);
\draw(0em,-0.6em)--(0.4em,-0.6em);
\draw[densely dotted, thick](1em,0)--(1em,-1em)--(0,-1em);
}}}

 \def\KL{\raisebox{-0.15em}{\tikz[scale=0.9]{\filldraw [black, fill opacity=0.4](1em,0)--(0,0)--(0em,-1em); \draw[densely dotted, thick](1em,0)--(1em,-1em)--(0,-1em);
}\,}}

 \def\Square{\raisebox{-0.15em}{\tikz[scale=0.9]{\draw(0,0)--(1em,0)--(1em,-1em)--(0,-1em)--cycle;
}}}

 \def\diag{\raisebox{-0.15em}{\tikz{\draw(0,0)--(1em,-1em);
}}}

 \def\UR{\raisebox{-0.15em}{\tikz[scale=0.9]{\filldraw [black, fill opacity=0.4](0,0)--(1em,0)--(1em,-1em); \draw[densely dotted, thick](1em,-1em)--(0,-1em)--(0,0);
}\,}}

 \def\CC{\raisebox{0.15em}{\tikz{\draw(0.5em,-0.5em)--(1em,0);
}}}

 \def\UP{\raisebox{-0.15em}{\tikz[scale=0.9]{\filldraw [black, fill opacity=0.4](1em,0)--(0,0)--(0.5em,-0.5em); \draw [densely dotted, thick] (1 em,0)--(1 em,-1em)--(0.5em,-0.5em);
}\,}}

 \def\RP{\raisebox{-0.15em}{\tikz[scale=0.9]{\draw [densely dotted](0,0)--(1em,0)--(0.5em,-0.5em)--cycle; \filldraw [black, fill opacity=0.4] (1 em,0)--(1 em,-1em)--(0.5em,-0.5em);
}\,}}

 \def\UPC{\raisebox{0.15em}{\tikz[scale=1.4]{\filldraw[black, fill opacity=0.4](0.8em,0)--(0,0)--(0.4em,-0.4em); \draw(0.6em,-0.4em)--(1em,0 em);
}}}

 \def\RPC{\raisebox{-0.15em}{\tikz[scale=1.4]{\filldraw[black, fill opacity=0.4](0.8em,0)--(0.8em,-0.8em)--(0.4em,-0.4em);\draw(0.3em,-0.3em)--(0.8em,0.2 em);
}\,}}

 \def\pUP{\raisebox{-0.25em}{\tikz[scale=0.9]{\draw(0,0)--(1.2em,0)--(0.6em,-0.6em)--cycle; \draw(0.15em,-0.15em)--(1.05em,-0.15em);
\draw(0.3em,-0.3em)--(0.9em,-0.3em);
\draw(0.45em,-0.45em)--(0.75em,-0.45em);
 \draw [densely dotted, thick] (1.2 em,0)--(1.2 em,-1.2em)--(0.6em,-0.6em);
}\,}}

 \def\pUPt{\,\raisebox{-0.2em}{\tikz[scale=0.42]{\draw(0,0)--(1.2em,0)--(0.6em,-0.6em)--cycle; \draw(0.15em,-0.15em)--(1.05em,-0.15em);
\draw(0.3em,-0.3em)--(0.9em,-0.3em);
\draw(0.45em,-0.45em)--(0.75em,-0.45em);
 \draw [densely dotted, thick] (1.2 em,0)--(1.2 em,-1.2em)--(0.6em,-0.6em);
}\,}}

 \def\pUPs{\,\raisebox{-0.2em}{\tikz[scale=0.5]{\draw(0,0)--(1.2em,0)--(0.6em,-0.6em)--cycle; \draw(0.15em,-0.15em)--(1.05em,-0.15em);
\draw(0.3em,-0.3em)--(0.9em,-0.3em);
\draw(0.45em,-0.45em)--(0.75em,-0.45em);
 \draw [densely dotted, thick] (1.2 em,0)--(1.2 em,-1.2em)--(0.6em,-0.6em);
}\,}}

 \def\UPs{\raisebox{-0.25em}{\,\tikz[scale=0.7]{\filldraw [black, fill opacity=0.4](1em,0)--(0,0)--(0.5em,-0.5em); \draw [densely dotted, thick] (1 em,0)--(1 em,-1em)--(0.5em,-0.5em);
}\,}}

 \def\UPCs{\raisebox{0.1em}{\,\tikz[scale=0.7]{\filldraw[black, fill opacity=0.4](0.8em,0)--(0,0)--(0.4em,-0.4em); \draw(0.6em,-0.4em)--(1em,0 em);
}}}

 \def\URs{\raisebox{-0.1em}{\,\tikz[scale=0.6]{\filldraw [black, fill opacity=0.4](0,0)--(1em,0)--(1em,-1em); \draw[densely dotted, thick](1em,-1em)--(0,-1em)--(0,0);
}\,}}

 \newcommand{\on}{\operatorname}
 \newcommand{\mc}{\on{main}}
\newcommand{\lmc}{\on{l.\! main}}
\newcommand{\rmc}{\on{r.\! main}}
\newcommand{\minc}{\on{minor}}

\newcommand{\core}{\on{core}}

\newcommand{\Limb}{\on{Limb}}

\def\leq{\leqslant}

\def\p{\mathfrak p}

\def\f{\mathfrak f}

\def\C{\mathbb C}

\def\F{\mathbb F}

\def\cO{\mathcal O}

\def\cA{\mathcal A}

\def\cL{\mathcal L}
\def\cR{\mathcal R}
\def\cE{\mathcal E}

\def\cD{\mathcal D}
\def\cV{\mathcal V}

\def\fX{\mathfrak X}

\def\fM{\mathfrak M}
\def\fD{\mathfrak D}

\def\fC{\mathfrak C}

\def\fB{\mathfrak B}
\def\fA{\mathfrak A}

\def\fm{\mathfrak m}

\def\Pl{\mathbb{PL}}

\DeclareMathOperator{\Hom}{Hom}

\DeclareMathOperator{\Res}{res}
\DeclareMathOperator{\RRes}{Res}
\DeclareMathOperator{\Ind}{Ind}

\DeclareMathOperator{\row}{row}

\DeclareMathOperator{\Stab}{Stab}

\DeclareMathOperator{\main}{main}

\DeclareMathOperator{\verge}{verge}
\DeclareMathOperator{\suppl}{suppl}

\DeclareMathOperator{\supp}{support}

\DeclareMathOperator{\Mat}{Mat}

\begin{document}

%%%%%%%%%%%%%%%%%%%%%%%%%%%%%%%%%%%%%%%%%%%%%%%%%%%%%%%%%%%%%%%%%%%%%%%%%%
%\title{Classification of the monomial orbit modules for \\  $p$-Sylow subgroups of finite classical groups}
\title{Orbit method for $p$-Sylow subgroups of finite classical groups}

\author{Qiong Guo$^{*}$, Markus Jedlitschky$^{**}$, Richard Dipper$^{**}$\\ \\$^{*}${\footnotesize College of Sciences, Shanghai Institute of Technology} \\ {\footnotesize 201418 Shanghai, PR China}
%\\ \scriptsize{E-mail:qiongguo@hotmail.com}
\smallskip \\$^{**}$ {\footnotesize Institut f\"{u}r Algebra und Zahlentheorie}\\ {\footnotesize Universit\"{a}t Stuttgart, 70569 Stuttgart, Germany}
%\\ \scriptsize{E-mail:  markus@mathematik.uni-stuttgart.de, Richard.dipper@mathematik.uni-stuttgart.de}\\
\setcounter{footnote}{-1}\footnote{\scriptsize E-mail: qiongguo@hotmail.com, markus.jedlitschky@gmail.com, richard.dipper@mathematik.uni-stuttgart.de}
\setcounter{footnote}{-1}\footnote{%\emph{Date:} \today}
%\color{red}
{\scriptsize This work was  supported by NSFC no.11601338.}}
\setcounter{footnote}{-1}
\setcounter{footnote}{-1}}
\date{}

%%%%%%%%%%%%%%%%%%%%%%%%%%%%%%%%%%%%%%%%%%%%%%%%%%%%%%%%%%%%%%%%%%%%%%%%%%

\maketitle

\begin{center}
\today
\end{center}

\begin{abstract}

For the $p$-Sylow subgroups $U$ of the finite classical groups of untwisted Lie type, $p$ an odd prime,  we construct a monomial $\C U$-module $M$ which is isomorphic to the regular representation of $\C U$ by a modification of Kirillov's orbit method called monomial linearisation. We classify a certain subclass of orbits of the $U$-action on the monomial basis of $M$ consisting of so called staircase orbits and show, that every orbit module in $M$ is isomorphic to a staircase one. Finally we decompose the Andr\'{e}-Neto supercharacters of $U$ into a sum of $U$-characters afforded by staircase orbit modules contained in $M$.   
\end{abstract}

\section{Introduction}

In the early sixties of the last century, A. Kirillov devised his orbit method (now known also as Kirillov theory) for his investigation of irreducible unitary representations of nilpotent and other classes of Lie groups.
 Starting point in this approach is the observation, that the action of Lie group on the dual of its Lie algebra provides a lot of information on representations of the group. For example for nilpotent groups Kirillov established a correspondence between irreducible unitary representations and the orbits (called coadjoint orbits) of the group acting on the dual of its Lie algebra. For a good exposition of this theory, see [\cite{kirillov}].

C. Andr\'{e} modified in a long series of paper [\cite{andre}-\cite{andre6}] the orbit method for his investigation on the characters of finite general unitriangular groups $U_n(q)$. It is known that the classification of the conjugacy classes and irreducible characters of $U_n(q)$ simultaneously for all $q$ and $n$ is a wild problem in the categorical sense. Indeed by [\cite{AriasetAl}] they are considered unknowable. More precisely 
Gudivok et al  showed in [\cite{gudivok}] that a nice description of the conjugacy classes leads to a nice
description of wild quivers. Hence it seems to be a good idea to replace the problem with an easier, doable one. Andr\'{e} just did this: His basic characters share important properties with irreducible ones and he classified them. 
 
 N. Yan developed in [\cite{yan}] a completely algebraic and combinatorial construction of Andr\'{e}'s basic characters. He showed that these are afforded by  orbits of $U_n(q)$ acting monomially on the space $\hat{\mathfrak u}$ of linear characters of its Lie algebra $\mathfrak u$. In addition he investigated biorbits of $U_n(q)$ acting on $\mathfrak u$, yielding certain unions of conjugacy classes of $U_n(q)$. Indeed he constructed, what nowadays is called a supercharacter theory, a term which was coined in [\cite{super}] by Diaconis and Isaacs. Such a theory consists of a set of characters, called supercharacters, and a set of unions of conjugacy classes, called superclasses, such that each irreducible character is constituent of precisely one supercharacter, each conjugacy class is contained in precisely one superclass, the superclasses and supercharacters are in one by one correspondence and supercharacters are constant on superclasses. 
 
To extend directly Andr\'{e}'s or Yan's method to the $p$-Sylow subgroups $U$ of other finite classical groups  does not work, basically since these are not algebra groups. Andr\'{e} and Neto  hence defined in [\cite{andreneto1},\cite{andreneto2},\cite{andreneto3}] supercharacter theories for $U$ of  of Lie type of type $\fB_n, \fC_n, \fD_n$  by restricting certain supercharacters from the overlying full unitriangular group $U_N(q)$ to $U_n$, (see e.g. [\cite{andreneto2}, 3.4]), respectively intersecting superclasses of $U_N(q)$ with $U_n$.  Here we set $N=2n$ for Dynkin types $\fC_n, \fD_n$ and $N=2n+1$ for type $\fB_n$.  More recently in  [\cite{andrefreitasneto}]  Andr\'{e}, Freitas  and Neto  and in [\cite{andrews}]  Andrews generalized this work extending it among other things to all finite classical groups of Lie type in a uniform way.
  
 Andr\'{e}-Yan supercharacters for $U_N(q)$ are afforded by orbit modules of a monomial action of $U_N(q)$ on $\hat{\mathfrak u}$.  The subgroup $U$ of $U_N(q)$ does in general not act on the latter transitively anymore. The main goal of this paper is to decompose the restrictions to $U$ of the relevant monomial $U_N(q)$-orbit modules into a direct sum of monomial $U$-orbit modules. 

In order to establish this main result we apply the concept of monomial linearisation of a group, introduced by the second named author in his doctoral thesis [\cite{markus}]. It may be considered as a modification of the original orbit method of Kirillov. The basic idea is to exhibit a special basis of the group algebra $\C G$ on which $G$ acts monomially by right multiplication exhibiting many orbits. So we may decompose the regular representation $\C G$ into a direct sum of many orbit modules. 

In this paper we obtain a monomial linearisation of the regular representation for the $p$-Sylow subgroups $U$ of finite classical groups of untwisted Lie type $\fB_n, \fC_n, \fD_n$, where $p$ is the characteristic of their underlying field and, throughout, is different from 2. 
%We define ``coadjoint" orbits of these groups by dualizing their action as vector space automorphisms on certain spaces $V$, which are different from, but closely related to the Lie algebra of $U$. 
We will not obtain a full classification of the orbits of $U$ acting on the monomial basis of $\C U$.
%on the dual space of $V$, and hence on the $\C$-basis of linear characters of the additive group $(V, +)$. 
However we will show that there exists a certain subset of the collection of the orbits, called staircase orbits, such that every orbit module is isomorphic to the orbit module of a staircase orbit. Then we shall produce a classification of the staircase orbits. In the final section we shall construct  the elementary Andr\'{e}-Neto supercharacters as characters afforded by certain unions of our orbits. From this we finally obtain the desired decomposition of the relevant monomial $U_N(q)$-orbits into a disjoint union of $U$-orbits.

This paper is based on ideas developed by the second named author in his doctoral thesis [\cite{markus}], containing part of the results there. Here we provide among other things in particular a generalization to all classical groups of untwisted Lie type.

\section{$p$-Sylow subgroups of finite classical groups}

Let $p$ be an odd prime and $q$ be some power of $p$ and let $U$ be a $p$-Sylow subgroup of classical group $G(q)$ over $\F_q$. We restrict ourselves here to groups $G(q)$ of untwisted Lie type, that is of  type $\fB$, $\fC$ and $\fD$ defined over $\F_q$. For our purpose it is convenient to set up these groups as subgroups of the canonical overlying full unitriangular groups.  More precisely  let $n$ be a natural number and set  $N=2n+1, \tilde n = n+1$ for Dynkin type $\fB_n$ and $N=2n, \tilde n = n$ for types $\fC_n, \fD_n$. Let $G=G_n(q)$ be a finite group of Lie type  $\fB_n$, $\fC_n$ or $\fD_n$ defined over $\F_q$. We identify $G$ with a corresponding orthogonal or symplectic group contained in $GL_N(q)$ and construct the $p$-Sylow subgroup $U$ of $G$ as subgroup $G\cap U_N(q)$ of the group $U_N(q)$ of upper unitriangular $N\times N$-matrices over $\F_q$.  

For any rectangular matrix $A$, we denote the $(i,j)$-th entry of $A$ by $A_{ij}$, the transpose of $A$ by $A^t$  and the $(i,j)$-th matrix unit by $e_{ij}$.  Moreover $\supp(A)=\{(i,j)\,|\, A_{ij}\neq 0\}$. Thus
$$
A = \sum_{(i,j)\in\supp(A)}A_{ij}e_{ij}
$$

We let $\epsilon = -1$, if $G$ is of type $\fC_n$ and $n+1\leq i\leq N$ and set $\epsilon = 1$ otherwise. We define the {\bfseries mirror map}\index{mirror map}
$\bar{}:\; \{1, \dots, N\} \rightarrow \{1, \dots, N\}:\; i \longmapsto \bar i := N+1-i$
which mirrors every entry on $\frac{N+1}{2}$. It satisfies $\bar{\bar i} = i$, $ i < j \Longleftrightarrow \bar i > \bar j$, and $i =  j  \Longleftrightarrow \bar i = \bar j$ for all  $i,j \in \{1, \dots, N\}$. Setting 
$$
S=  \sum_{i=1}^n e_{i\bar i} +  \sum_{i=n+1}^N\epsilon e_{i\bar i}
$$

we define the bilinear form $\langle  \, , \,  \rangle: \F_q^N\times \F_q^N \rightarrow \F_q: (v,w)\rightarrow \langle v,w\rangle=v^t S w$, where the vectors of $\F_q^N$ are written as column vectors. Then $S$ is the Gram matrix of $\langle\,,\,\rangle$ with respect to the natural basis $\cE_N=\{e_1,\ldots, e_N\}$ of $\F_q^N$, that is $S_{ij}=\langle e_i, e_j \rangle$. Since $S$ is invertible, 
$\langle\,,\,\rangle$  is nondegenerate and is  symmetric for Lie types $\fB_n, \fD_n$ and symplectic for type $\fC_n$.

Then the corresponding classical group  $G$ is given as group of invertible $N\times N$ matrices leaving $\langle\,,\,\rangle$ invariant. 

For $A\in \Mat_{N\times N}(q)$ we define $A^R=S^{-1}A^tS$,
then $A\mapsto A^R$ is an $\F_q$-algebra antiautomorphism  of $\Mat_{N \times N}(q)$, the algebra of $N\times N$-matrices, and  $g\in GL_n(q)$ is contained in $G$ if and only if $g^R=g^{-1}.$

For $1\leqslant i, j\leqslant N$, we define $\varepsilon_{ij}=S_{\bar i i}S_{\bar j j}$. Then $\varepsilon_{ij}=1$ unless $G$ is of type $\fC_n$ and $1\leq i\leq  n$ and $n+1\leq j \leq N$ or $n+1 \leq i \leq N$ and $1\leq j\leq n$. In this case we have $\varepsilon_{ij} = -1$.

A direct calculation lets us  express the map $-^R$ explicitly on the matrix entries:

\begin{Lemma}\label{description of R}
Let $A$ be a $N\times N$-matrix. Then   $(A^R)_{ij}=\varepsilon_{ij} A_{\bar j\, \bar i}$ for all $1\leqslant i, j\leqslant N$. Moreover
for types $\fB_n, \fD_n$ , the matrix  $A^R$ is obtained from $A$ by reflecting the entries of $A$ along the antidiagonal and in type $\fC_n$ this holds up to a sign. In particular, $\big(U_N(q)\big)^R=U_N(q)$. Thus,  if $A$ is $R$-invariant, then $A_{ij}=0$ if and only if  $A_{\bar j\, \bar i}=0$.\hfill $\square$

\end{Lemma}

\begin{Remark}
For $A\in \Mat_{N\times N}(q)$, $AA^R$ is $R$-invariant and hence in particular the last assertion of the lemma above applies to it.\hfill $\square$
\end{Remark}

In the following we frequently work with various linear subspaces of  $\Mat_{N\times N}(q)$ defined as sets of all matrices being supported in special subsets of $\{(i,j)\,|\, 1\leq i,j\leq N\}$. To help with the bookkeeping of those we shall use specially designed descriptive symbols for these subsets:

\begin{Notation}
\quad
\begin{itemize}
\item [1)] Let $\Square$  $=$  $\{ (i,j) \; | \; 1 \leq i,j \leq N \}=\tilde \Phi,$
and $\UR=$  $\{ (i,j) \in \Square \; | \; i < j \}=\tilde \Phi^+$.
\item [2)] We denote the diagonal by $\diag :=\{(i,j) \in \Square \; | \; i=j \}$ and 
half of the antidiagonal by ${\CC}:=\{ (i,j)\in \UR \; | \;  i = \bar j \}$. 
\item [3)] Let  ${\UP}= \{ (i,j) \in \Square \; | \; i < j < \bar i \}$ and 
${\RP}=\{ (i,j)\in \Square \; | \; \bar j < i < j \}$, illustrated as follows:
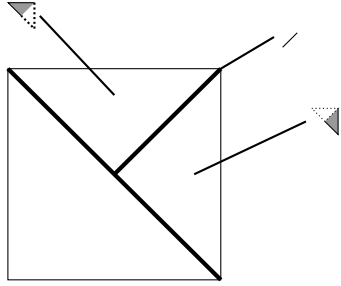
\begin{center}
\begin{tikzpicture}[scale=0.7]

\draw (0,-2) rectangle (4,2);
%\path[fill=blue] (2,0)--(4,2)--(0,2)--cycle;
%\path[fill=red] (2,0)--(4,2)--(4,-2)--cycle;
\draw [color=black, ultra thick] (2,0)--(4,2);
\draw [ultra thick] (0,2)--(4,-2);
\draw (6,1) node {${\RP}$};
\draw [thick] (3.5,0)--(5.6,1);

\draw (0.3,3) node {${\UP}$};
\draw [thick] (2,1.5)--(0.6,3);

\draw [thick] (4,2)--(5,2.6);
\draw (5.3,2.5) node {$\CC$};
\end{tikzpicture}
\end{center}
Thus $\UR=\UP\cup \CC \cup \RP$.
\item[4)] Set $\UPC=\UP\cup \CC$ and $\RPC=\RP\cup \CC$.
\item[5)] To simplify notation we define $\pUP$ to be $\UP$ for types  $\fB_n, \fD_n$ and $\UPC$ for type $\fC_n$.
\end{itemize}\hfill $\square$
\end{Notation}

\begin{Lemma}\label{lemmaurs} Let $u\in U_N(q)$. Then $u\in U=G\cap U_N(q)$  if and only if $(uu^R)_{rs}=0$  for all $(r,s)\in \RPC$ and then for all $(r,s)\in \RPC$:

\begin{equation}\label{urs}
u_{rs}=-\varepsilon_{rs} u_{\bar s\, \bar r}-\sum_{r<l<s}\varepsilon_{ls} u_{rl}u_{\bar s\, \bar l}.
\end{equation}
\end{Lemma}

\begin{proof}
Recall that $u\in U$ if and only if $u^{-1}=u^R$ or equivalently 
$uu^R=u^Ru=1$.
Since $u\in U_N(q)$, we have $(uu^R)_{ii}=1$ for $1\leqslant i\leqslant N$
Since $uu^R=(uu^R)^R$ we have $(uu^R)_{ij}=\varepsilon_{ij}(uu^R)_{\bar j \, \bar i} $ by \ref{description of R}, hence $(uu^R)_{ij}=0 $ if and only if $(uu^R)_{\bar j \, \bar i}=0 $ for $1\leqslant i<j \leqslant N$. Thus $u\in U$ if and only if $(uu^R)_{rs}=0$ for all $(r,s)\in \RPC$. Thus suppose $u\in U$ and let $(r,s)\in \RPC$. Then we obtain from \ref{description of R}  using the fact that $u, u^R\in U_N(q)$:
\begin{eqnarray*}
0=(uu^R)_{rs}&=&\sum_{l=1}^N u_{rl} (u^R)_{ls}
=\sum_{l=r}^s u_{rl}(u^R)_{ls}\\
&=& u_{rs}+(u^R)_{rs}+\sum_{r<l<s} u_{rl}(u^R)_{ls}\\
&=& u_{rs}+\varepsilon_{rs}u_{\bar s \, \bar r}+ \sum_{r<l<s} \varepsilon_{ls} u_{rl}u_{\bar s\, \bar l}.
\end{eqnarray*}
Thus $u_{rs}=-\varepsilon_{rs} u_{\bar s\, \bar r}-\sum_{r<l<s}\varepsilon_{ls} u_{rl}u_{\bar s\, \bar l}$. In addition, if $r<l<s$, then $\bar r>\bar l >\bar s$  and hence $(\bar s,\bar l)\in \UP$ since $(\bar s,\bar r)\in \UPC$.
\end{proof}

%Now it can be seen easily from \ref{urs} that:
\begin{Cor}\label{depend on the left}
Let $u\in U$ and $(r,s)\in \RPC$. Then  $u_{rs}$ is determined by the  entries on the positions of row $r$ to the left of $(r,s)$ or of row $\bar s$ to the left of or on  $(\bar s, \bar r)$.\hfill$\square$
\end{Cor}

We iterate the argument and obtain recursively:

\begin{Theorem}\label{unique}

If $G$ is of type $\fB_n$ or $\fD_n$ let $(r,s)\in \RPC$, and if it is of type  $\fC_n$, then let $(r,s)\in \RP$. Then there exists a polynomial $p_{rs}(t_{ij})$ in variables $t_{ij}$ with $(i,j)\in \pUP$  such that $u_{rs}=p_{rs}(u_{ij})$. Moreover, for each $(i,j)\in \pUP$  choose $\lambda_{ij}\in\F_q$. Then there exists a  unique element $u\in U$ such that $u_{ij}=\lambda_{ij}$ for all $(i,j)\in \pUP$.
\end{Theorem}
\begin{proof}
Suppose $G$ is of type $\fB_n$ or $\fD_n$. Then we have $\varepsilon_{ij}=1$ for all $1\leqslant i, j\leqslant N$. 

Let $(r, \bar r)\in \CC$. Then \ref{urs} becomes:
\begin{equation}\label{urbarr}
u_{r\bar r}=-\frac{1}{2}\sum_{r<l<\bar r}u_{rl}u_{r\, \bar l}
\end{equation}
Since $r<l<\bar r$ we have $r<\bar l <\bar r$ and hence $(r,l),(r, \bar l)\in \UP$. Thus, 
\begin{equation}\label{prbarr}
p_{r\bar r}=-\frac{1}{2}\sum_{r<l<\bar r}t_{rl}t_{r\, \bar l}
\end{equation}
 is the desired polynomial for  $(r, \bar r)\in \CC$. 

Now assume $(r,s)\in \RP$ then $ (\bar s, \bar r)\in \UP$ and hence $(\bar s, \bar l)\in \UP$ for all $r<l<s$. Moreover the positions $(r,l)$ in \ref{urs} are to the left of $(r,s)$. Inductively  we may assume that we have for all the positions $(r,l)$ already defined polynomials $p_{rl}$ in variables $t_{ij}$ with $(i,j)\in \UP$ such that $p_{rl}(t_{ij})=u_{rl}$. Then, setting 
 \begin{equation}\label{prs+}
p_{rs}=- t_{\bar s\, \bar r}-\sum_{r<l<s}t_{\bar s\, \bar l}p_{rl}
\end{equation}
we have obtained the polynomials we want. 
Moreover this proves as well that we can choose  $\lambda_{ij}\in\F_q$ for$(i,j)\in \UP$ freely  and obtain values  for the positions $(r,s)$ with $(r,s)\in \RPC$ using formulas \ref{prs+} and \ref{prbarr} to produce $u\in U$ satiesfying $u_{ij} = \lambda_{ij}$ for all $(i,j)\in \UP$.
If $G$ is of type $\fC_n$ similar arguments apply observing that  $\varepsilon_{ij} = -1$ occurs. \end{proof}

\begin{Remark} We remark in passing, that theorem \ref{unique} shows among other things, that the groups $U$ are indeed $p$-Sylow subgroups of the finite classical groups. This follows immediately from the well known order formulas for the finite classical groups, (see e.g. [\cite{grove}]), and the immediate consequence of \ref{unique} given as $|U|=q^a$ with $a=|\pUP|$.
\hfill$\square$ \end{Remark}

\iffalse

For later use we need a somewhat preciser statement, which is immediate from the proof of \ref{unique}:
\begin{Remark}\label{uniqueCor}
Note that the equations in \ref{unique} also state, that we can express the entry at position $(r,s)$ of matrix  $u\in U$ in the illustration \ref{dependleftresult} below as  polynomial  in the entries at positions in
$R_{rs}:= \{ (i,j) \in \UPC \; | \; \bar s \leq i \leq r \text{ and } j \leq \bar r \}$ with coefficients in the prime field $\F_p$. These positions are colored  green in \ref{dependleftresult} below. Moreover note that the polynomial attached to position $(\bar s, \bar r)$ is the constant polynomial $\pm 1$.
\hfill$\square$
\end{Remark}
\begin{equation}\label{dependleftresult} 
\includegraphics[width=0.4\textwidth]{./GeoHom/dependleft3}
\quad\quad
\includegraphics[width=0.4\textwidth]{./GeoHom/dependleft4}
\end{equation}
\fi
 
\begin{Remark}\label{directc} Direct inspection of equation  \ref{urs} in  lemma \ref{lemmaurs} yields in particular the following:  If $u_{sr}\neq 0$ and $u_{ij}=0$ for all $(s,r)\neq (i,j)\in \UP$,  then $u_{\bar s \, \bar r}=-\varepsilon_{rs}u_{rs}$ and $u_{ij}=0$ for all $(i,j)\in \UR$ with $(i,j)\neq (r,s), (\bar s, \bar r)$ for $(r,s)\in\UP$ , provided
  $G$ is not of type $\fB_n$ or $r\neq n+1$. If $G$ is of type $\fB_n$, and $r= n+1$,  then $u$ has one more additional nonzero entry, namely $u_{s\bar s} = -\frac{1}{2} (u_{ s, n+1})^2$.
  
\hfill$\square$
\end{Remark}

From standard arguments from Lie theory one sees easily, that for types $\fB_n, \fC_n, \fD_n$  the positive roots $\Phi^+$ of the associated root system $\Phi$ are in bijection with the set
$\pUP$ (or rather pairs $\{(i,j), (\bar j, \bar i)\}, 1<i<j\leq \tilde n$, 
 for details see e.g. [\cite{carter}, section 11.2]). Hence we may identify $\Phi^+$ and $\pUP$. For type $\fA_{N-1}$ we identify $\Phi^+$ with $\UR$.

We call $J\subseteq\Phi^+$ {\bf closed}, if $\alpha, \beta\in J$ and $ \alpha+ \beta\in \Phi^+$ implies $\alpha+\beta\in J$. Translating this into a statement on subsets $J$ of $\pUP$ (respectively of $\UR$) one proves by direct calculation:

\begin{Lemma}\label{Upattern}
 Let $J \subseteq \pUP$ be a set satisfying 
\begin{align*} (i) && (i,j), (j,k) \in J &\quad\quad \Rightarrow \quad\quad (i,k) \in J  \\ \text{and} \quad (ii) && (i,j),(\bar k, \bar j) \in J, \; (i,k) \in \pUP &\quad\quad \Rightarrow \quad\quad (i,k) \in J. \end{align*}
Then $J$ is a {closed subset} of $\pUP$.  In type $\fA_{N-1}$, $J\subseteq \UR$ is closed if condition $ (i) $ is satisfied. Note that $\pUP$ is closed in $\UR$, that is $(i,j), (j,k)\in \pUP$ implies $(i,k)\in \pUP$.\hfill $\square$
\end{Lemma}

For type $\fA_{N-1}$ we denote $\Phi$ henceforth  by $\tilde \Phi=\tilde \Phi^+\cup \tilde \Phi^-$, where $\tilde \Phi^-=\{(i,j)\in \tilde \Phi\,|\,i>j\}$.

\begin{Lemma}\label{rootsubgroupsU}
For $1\leqslant i<j \leqslant N$ and $\alpha\in \F_q$ let $\tilde x_{ij}(\alpha)=1+\alpha e_{ij}\in U_N(q)$ and $\tilde X_{ij}=\{\tilde x_{ij}(\alpha)\,|\, \alpha\in \F_q\}$, then $\tilde x_{ij}(\alpha)\tilde x_{ij}(\beta)=\tilde x_{ij}(\alpha+\beta)$ for $\alpha, \beta\in \F_q$, and hence $\tilde X_{ij}\cong (\F_q, +)$. These are the root subgroups of $U_N(q)$ of type $\fA_{N-1}$. Now
Let $(i,j) \in \pUP$. We define  \vspace{-.4cm}
\begin{center}\begin{tabular}{p{.14\textwidth}p{.49\textwidth}p{.33\textwidth}}
in type $\fB_n$: &  $x_{ij}(\alpha) = 1+\alpha e_{ij}- \alpha e_{\bar j \, \bar i} 
=\tilde x_{ij}(\alpha)\tilde x_{\bar j\, \bar i}(-\alpha)$ &  where $\alpha \in \F_q$,\; if $j \ne n+1$, \\[.2cm]
& $x_{i,n+1}(\alpha) = 1+\alpha e_{i, n+1} - \alpha e_{n+1,\bar i} - \frac{1}{2}\alpha^{2}e_{i, \bar i}$ &\\[.2cm]
& $\phantom{x_{i,n+1}(\alpha)}=\tilde x_{i,n+1}(\alpha)\tilde x_{n+1,\, \bar i}(-\alpha)\tilde x_{i \bar i} (\frac{1}{2}-\alpha^2)$ &   where $\alpha \in \F_q$, \\
& $\phantom{x_{i,n+1}(\alpha)}$\\[-.1cm]
in type $\fC_n$: &  $x_{ij}(\alpha) = 1+\alpha e_{ij}- \alpha e_{\bar j \,\bar i}=\tilde x_{ij}(\alpha)\tilde x_{\bar j\, \bar i}(-\alpha) $  &  where $\alpha \in \F_q$,\; if $j \leq n$, \\[.1cm]
&  $x_{ij}(\alpha) = 1+\alpha e_{ij} + \alpha e_{\bar j \,\bar i}=\tilde x_{ij}(\alpha)\tilde x_{\bar j\, \bar i}(\alpha)$ &  where $\alpha \in \F_q$,\; if $n < j <  \bar i$, \\[.1cm]
& $x_{i\bar i}(\alpha) = 1+\alpha e_{i \bar i}=\tilde x_{i \bar i}(\alpha)$ & where $\alpha \in \F_q$, \\[.3cm]
in type $\fD_n$: &  $x_{ij}(\alpha) = 1+\alpha e_{ij}- \alpha e_{\bar j \,\bar i}=\tilde x_{ij}(\alpha)\tilde x_{\bar j\, \bar i}(-\alpha) $ &  where $\alpha \in \F_q$.
\end{tabular}
\end{center}
   We define $X_{ij}=\{x_{ij}(\alpha)\,|\, \alpha\in \F_q\}$. Then $X_{ij}\cong (\F_q, +)$ is the root subgroup of $U$ associated to the position $(i,j)\in \pUP$.
\end{Lemma}
\begin{proof}
This follows immediately from \ref{directc}, (comp. [\cite{carter}, Section 11.3]).
\end{proof}

\begin{Defn}\label{def of generator}
 Let $J \subseteq \pUP$ (resp. $J\subseteq \UR$) be closed.
The {\bfseries pattern subgroup} $U_J$ (resp. $\widetilde U_J$) is defined to be a subgroup of $U$ (resp. $U_N(q)$) generated by all root subgroups $X_{ij}$ (resp. $\tilde X_{ij}$) with $(i,j) \in J$. Then in particular $U_N(q) = \widetilde U_{_ {\widetilde \Phi^+}}$. We denote from now on $U_N(q)$ by $\widetilde U$.
\hfill$\square$
\end{Defn}

The following result is well known (see e.g. [\cite{carter}]):

\begin{Theorem}\label{anyorder}
Let $J$ be defined as in \ref{def of generator} and fix  an arbitrary linear ordering on $J$. Then each $u \in U_J$ (resp. $\widetilde U_J$) can be \emph{uniquely} written as a product of ${x}_{ij}(\lambda)$'s  (resp. $\tilde x_{ij}(\lambda)$'s ) , where $(i,j)$ runs through $J$ and $\lambda$ runs through $\F_q$ with the product taken in that fixed order.\hfill$\square$
\end{Theorem}

Obviously, if $J\subseteq \pUP$ is closed, then $J$ satisfies in particular condition (i) of \ref{Upattern}, and hence is closed in $\UR$ too. 

\begin{Defn}\label{patternUtotildeU}
Let $J\subseteq \pUP$. Then $\widetilde J$ denotes the set of positions of $J$ considered as subset of $\UR$ and (abusing notation) $\widetilde U_J$ is the corresponding pattern subgroup of $\widetilde U$ of type $\fA$. Note that obviously $\widetilde\pUP$ is closed in $\UR$ and hence $\widetilde U_{\pUPs}$ is a pattern subgroup of $\widetilde U$. Moreover, $\widetilde U_J$ is always contained in $\widetilde U_{\pUPs}$ for all closed subsets $J$ of $\pUP$.
\hfill$\square$
\end{Defn}

\section{Monomial linearisation of $\C U$}

In this section we shall construct a basis of the group algebra $\C U$ on which $U$ acts monomially. For the general background of monomial linearisation we refer the reader to [\cite{DG3, markus}]. We first need some notation and a few basic definitions:

\begin{Defn}\label{defiV}
For $S\subseteq \square$, define 
$V_S=\{A\in \Mat_{N\times N}(q)\,|\, \supp A\subseteq S\}$. Then  
$V_S=\bigoplus_{(i,j)\in S}\F_q e_{ij}$. Note that  $V_\square=\Mat_{N\times N}(q)$. For the Dynkin type $\fX_n=\fB_n, \fC_n$ or $ \fD_n$, set 
$V=V(\fX_n)=V_{\pUPs}= \bigoplus_{(i,j)\in \pUPs} \F_q e_{ij}$. Thus
\[
V(\fB_n)=V(\fD_n)= \bigoplus_{(i,j)\in \UPs} \F_q e_{ij} \quad \text{ and } \quad
V(\fC_n) = \bigoplus_{(i,j)\in \UPCs} \F_q e_{ij}.
\]\hfill$\square$
\end{Defn}

Note that the so defined vector space $V=V_{\pUPs}$ is in general not invariant under matrix multiplication by elements  of $U$ from the right.

Then $\kappa: V_\square \times V_\square \rightarrow \F_q: (A, B)\mapsto tr(A^tB)$ is a nondegenerate symmetric bilinear form on $V_\square$. For $S\subseteq \square, T=\square\setminus S$ we have 
$V_S^\bot=\{B\in V_\square\,|\, \kappa (A, B)=0,\,\forall\, A\in V_S\}=V_T$.
In addition $\kappa|_{V_S\times V_S}$ is also a nondegenerate symmetric bilinear form. By direct calculation we have:
\begin{Lemma}\label{traceform}
The bilinear form $\kappa$ satisfies   
\[
\kappa(A,B) = tr(A^tB) =  \sum_{(i,j)\in \square} A_{ij}B_{ij}=\sum_{(i,j)\in \supp A\,\cap\, \supp B} A_{ij}B_{ij}.
\]
and $\kappa(B^tA, C)=\kappa(A, BC)=\kappa(AC^t,B)$ for all $A,B, C\in V_\square$. Moreover, if $A, B\in V_\square$ with $\supp A\cap \supp B=\emptyset$, then $\kappa(A,B)=0$.\hfill$\square$
\end{Lemma}

For $S\subseteq \square$ denote the natural projection $V_\square\rightarrow V_S$ with kernel
$V_S^\bot=\{A\in V_\square\,|\, \supp A\cap S=\emptyset\}$ by $\pi_{_S}$. 

Note that $A\tilde x_{ij}(\alpha)$ for $A\in \Mat_{N\times N}(q), (i,j)\in \tilde \Phi$ and $\alpha\in \F_q$ is obtained by adding $\alpha$ times column $i$ to column $j$ in $A$. Similarly $\tilde x_{ij}(\alpha)A$ is obtained by adding $\alpha$ times row $j$ to row $i$ in $A$. Since $\widetilde U = U_N(q)$ is generated by the root subgroups $\tilde X_{ij}, (i,j)\in \tilde \Phi^+=\UR$, one proves easily that $\supp (Ag^t)\subseteq \pUP\cup \tilde  \Phi^- \cup\diag \text{ and }\supp (Ag)\subseteq \UR=\tilde \Phi^+$ for $A\in V_{\pUPs}, \, g\in U_N(q)$.

\begin{Prop}\label{UactsAuto} Let $V=V_{\pUPs}, \pi=\pi_{\pUPt}$. Then the map
\[ V \times \widetilde U \rightarrow  V:\; (A,u) \mapsto A. u := \pi(Au),  \]
defines a group action, where the elements of $\widetilde U$ act as ($\F_q$-vector space) automorphisms. 
\end{Prop}
\begin{proof}
We have to show $A. (uv)=(A. u). v$ for $A\in V$ and  $u, v\in \widetilde U$. Observe that  $\supp(Bv^t)\subseteq \pUP\cup \tilde \Phi^-\cup\diag $  for all $B\in V$ and $\supp (Au)\subseteq \UR$. Moreover $\pi(B)=B$ and $(\pUP\cup \tilde \Phi^-\cup\diag)\cap \UR=\pUP$, and hence we obtain using \ref{traceform}:
\begin{eqnarray*}
\kappa\big(B, A. (uv)\big)&=&\kappa\big(B,\pi( Auv)\big)=\kappa(B, Auv)\\
&=&
\kappa(Bv^t, Au)=\sum\nolimits_{i,j\in \pUPs} (Bv^t)_{ij}(Au)_{ij}\\&=&\kappa(Bv^t, \pi(Au))=\kappa(Bv^t, A.u)\\
&=&\kappa(B, (A. u)v)=\kappa(\pi(B), (A. u)v)=\kappa(B, \pi((A. u)v))\\
&=&\kappa\big(B, (A. u).v\big) \quad \text{   for all } B\in V.
\end{eqnarray*}
Since $\kappa$ is nondegenerate on $V$ we conclude $A. uv=(A.u). v$ as desired. Obviously $u$ acts $\F_q$-linearly on $V$, since $\pi$ is $\F_q$-linear, and hence $u$ acts as automorphism.
\end{proof}

\begin{Defn}\label{deff}
Define
$f: \widetilde U\rightarrow V$ to be the restriction of the projection map $\pi=\pi_{\pUPt}$ to $\widetilde U$.\hfill$\square$
\end{Defn}

\begin{Theorem}\label{fonecocycle}
$f$  satisfies $f(uv) = f(u).v+f(v)$ and hence is a right 1-cocycle  (see [\cite{markus,DG3}]) . Moreover $f$ is surjective and  $f|_U$ is bijective.
\end{Theorem}

\begin{proof}
Let $x, g\in \widetilde U$, $A\in V$. Then, since $A=\pi(A)$ and $\pi(x)=\pi(x-1)$, we have
 \begin{eqnarray*}
 \kappa(A, f(x). g)&=& \kappa\big(A, \pi(f(x)g)\big)=\kappa\big(A, f(x)g\big)\\&=& \kappa (Ag^t, f(x))=\kappa\big(Ag^t, \pi(x) \big)=\kappa\big(Ag^t, \pi(x-1) \big)\\&=&\kappa(Ag^t, x-1)=\kappa(A, xg-g)=
 \kappa\big(A, \pi(xg)-\pi(g)\big)\\&=&\kappa\big(A, f(xg)-f(g)\big),
 \end{eqnarray*}
using $\supp (Ag^t)\subseteq \pUP\cup \tilde  \Phi^-\cup\diag$, $\supp(x-1)\subseteq \UR=\tilde \Phi^+$ and hence $\supp (Ag^t)\cap \supp(x-1)\subseteq \pUP$.
 Thus, since $\kappa$ is nondegenerate on $V$, we conclude
 $f(x). g=f(xg)-f(g)$ and hence $ f(xg)=f(x). g+f(g),$ that is 
 $f$ is a right 1-cocycle. By theorem \ref{unique} it is easy to see $f$ is surjective and  that $f|_U$ is bijective.
\end{proof}

Henceforth we denote for any abelian group $H$ the set of complex linear characters of $H$ by $\hat H$. Thus in particular $\hat V$ is the set of linear characters of the additive group $(V,+)$.  For $A\in V$, the map $\tau_A=\kappa(A,-): V\mapsto \F_q: B\mapsto \kappa(A,B)\in \F_q$ defines an element of the dual space $V^*=\Hom_{\F_q}(V, \F_q)$ and the map $\tau: V\mapsto V^*:A\mapsto \tau_A\in V^*$ is an $\F_q$-isomorphism.
Note that 

\begin{equation}\label{tausubA}
\tau_A = \sum_{(i,j)\in \pUPt}A_{ij}e^*_{ij},
\end{equation}
where $e_{ij}^*$ denotes the $(i,j)$-th coordinate function on $V$ for $(i,j)\in\pUP$, dual to the matrix unit $e_{ij}$ in the natural basis of $V$. 

\bigskip

Throughout $\theta:\F_q\rightarrow \C^*$ denotes a fixed non trivial linear character  of the additive group $(\F_q,+)$. Then the map $\theta\circ-: V^*\rightarrow \hat V: \eta\mapsto \theta\circ \eta\in \hat V$ for any $\eta\in V^*$ is bijective and we have  
\begin{equation}\label{hatV}
\hat V=\{\chi_A=\theta\circ \tau_A\,|\, A\in V\}.
\end{equation}

Thus using \ref{tausubA} we have
\begin{equation}\label{actionofchisubA}
\chi_A(B) = \theta\circ\tau_A(B) = \theta\Big(\sum_{(i,j), (r,s)\in\pUPt}A_{ij}B_{rs}e_{ij}^*(e_{rs})\Big) = \prod_{(i,j)\in\pUPt}\theta(A_{ij}B_{ij}) = \theta(\kappa(A,B)).
\end{equation}
Now $\widetilde U$ acts on $\hat V$ by 
$\chi.u: A\mapsto \chi(A.u^{-1})$  for $u\in G, \chi\in \hat V$ and $A\in V$.
Note that identifying the group algebra $\C V$ with the $\C$-algebra $\C^V$ of maps from $V$ to $\C$ by $\tau\mapsto \sum _{v\in V}\tau(v)v$ for $\tau\in \C ^V$, the linear character $\chi\in \hat V$ is mapped to $\sum_{v\in V}\chi(v)v=|V|e_{\bar \chi}$, where $e_{\bar \chi}\in \C V$ is the primitive idempotent affording the complex conjugate linear character $\bar \chi\in \hat V$. Thus the group algebra $\C V\cong \C ^V$ has $\C$-basis $\hat V$, and hence $\C \hat V\cong \C V \cong \C^V$.

We define $f^*: \C^V\rightarrow \C^{\widetilde U}\cong \C \widetilde U$ by  $f^*(\tau)=\tau\circ f$ for $\tau\in \C^V$. Then it is not hard to see, that $f^* = f|_U^{-1}$ is injective with $f^*(\C \hat V) \cong f^*(\C^V) = \C U \subseteq \C\widetilde U$ is an $\F_q$-isomorphism. 
We apply [\cite{markus}, 2.1.35] and [\cite{DG3}] to obtain:

\begin{Theorem}\label{monomial} The full upper unitriangular group $\widetilde U = U_N(q)$ acts monomially on $\hat V = \hat V_{\pUPs}$, where the action of $u\in \widetilde U$ on $\chi_A\in\hat V$ for $A\in V$ is given as 
\[
\chi_{_A}u=z \chi_{_A}.u=z \chi_{_{A.u^{-t}}}=z {\chi_{\pi_{_{\pUPt}}(Au^{-t})}},
\]
with $z=\theta(\kappa(A, \pi_{_{\pUPt}}(u^{-1})))=\chi_{_{A}}(f(u^{-1}))\in \C^*$. Moreover the restriction of $f$ to $U$ is bijective. The restriction to $U$ of the monomial action of $\widetilde U$ on $\hat V$ induces a $U$-isomorphism between $\C\hat V$ and the right regular representation $\C U_{\C U}$ given by  $f^* = f|_U^{-1}$, where  
\[ 
f^*(\chi_{_A}) = \sum_{u\in U}\theta\circ\kappa(A,u)u\in \C U \quad \text{for} A\in V.
\]
In particular 
\[ 
\{ \sum_{u\in U}\theta\circ\kappa(A,u)u\,|\, A\in V\}\subseteq \C U
\]
is a monomial basis of the group algebra $\C U$.\hfill$\square$
\end{Theorem}

\begin{Remark}\label{permmonom} Note that the dot action $\chi.u: A\mapsto \chi(A.u^{-1})$  for $u\in \widetilde U, \chi\in \hat V$ and $A\in V$ above is precisely the permutation representation of $ \widetilde U$ underlying the monomial action on $\hat V$ of \ref{monomial}.\hfill$\square$
\end{Remark}

\begin{Remark}\label{complements}
A quick calculation reveals, that $\pUP$ and $J= \tilde{\Phi}^+\setminus \pUP$ are closed subsets of $\tilde{\Phi}^+$.  Thus $\widetilde U_{\pUPs}$ and $\widetilde U_J$ are pattern subgroups of $\widetilde U$. Observe that obviously $\ker f = \{u\in\widetilde U\,|\, f(u)=0\} = \widetilde U_J$. Morover, both $U$ and $\widetilde U_{\pUPs}$ are complements of $\widetilde U_J$ in $\widetilde U$, that is $\widetilde U=\widetilde  U_{\pUPs} \widetilde U_J=\widetilde U_J\widetilde U_{\pUPs}$ and  $\widetilde U= U \widetilde U_J=\widetilde U_J U$.
\hfill$\square$
\end{Remark}

 By [\cite{embed}] we have the following result:

\begin{Theorem} \label{ChatVIsInd}    As $\C \widetilde U$-module
 $\C \hat V\cong\Ind^{\widetilde U}_{\widetilde U_J} \C_{\widetilde U_J}$, where $\C_{\widetilde U_J}$ is the trivial $\C \widetilde  U_J$-module.\hfill$\square$
\end{Theorem}

\begin{Notation}\label{lide}
For $A\in V$ we identify from now on $\chi_{_{-A}} \in \C^V$ with $[A]:=\sum_{B\in V}\overline{ \chi_{_A}(B)} B\in \C V$. We set $[A]_{ij}=A_{ij}$ for $(i,j)\in \pUP$. Thus $\hat V=\{[A]\,|\,A\in V\} \subseteq \C V$.  Let $A\in V, u\in \widetilde U$. Then applying theorem \ref{monomial} we  get $[A].u=[\pi_{\pUPt}(Au^{-t})]$. Throughout we call the the elements $[A]\in \hat V$ linear characters or simply characters. However note that those are linear characters of the additive group $(V,+)$ and not of $U$.
 \hfill$\square$
\end{Notation}

 Let $1\leqslant i<j\leqslant N, \alpha\in \F_q$ and $A\in V$. Then by theorem \ref{monomial} we have 
\begin{equation}\label{rootelmaction}
[A]\tilde x_{ij}(\alpha)=\theta(\kappa(-A, \pi_{_{\pUPs}}(\tilde x_{ij}(\alpha)^{-1})))[B]=
\begin{cases}
\theta(\alpha A_{ij})[B] & \text{ if } (i,j)\in \pUP\\
[B] & \text{ otherwise.}
\end{cases}
\end{equation}
where $B=\pi_{_{\pUPs}} (A\tilde x_{ji}(-\alpha))$. But $A\tilde x_{ji}(-\alpha)$ is obtained from $A$ by adding $-\alpha$ times column $j$ to column $i$ in $A$. We have shown:
\begin{Prop}\label{TruncatedColumnOperationG}
Let $A\in V$ and let $\tilde x_{ij}(\alpha)\in \widetilde U$ with $1\leqslant i<j\leqslant N$ and $\alpha\in \F_q$. Then $[A].\tilde x_{ij}(\alpha)$ arises from $A$ by adding $-\alpha$ times column $j$ to column $i$ in $A$ and setting entries outside of $\pUP$ to zero. \hfill$\square$
\end{Prop}

We call the permutation action above  {\bf restricted column operation}. 

\begin{Remark}\label{compTypeA} The action of $\widetilde{U}$ on $\hat V$ in theorem \ref{monomial} yields part of the Andr\'{e}-Yan supercharacters. More precisely, the map $\tilde f:\widetilde U\rightarrow V_{\URs}: u\mapsto u-1$ is a (left and right) $1$-cocycle and  yields a monomial linearisation of $\widetilde U$ (see e.g. [\cite{DG3}]). Moreover there is a natural map from $\hat V_{\URs}$ to $\hat V_{\pUPs}$ given by restriction of maps. Now \ref{TruncatedColumnOperationG} implies that this map is $\C\widetilde U$-linear. In addition $\hat{\widetilde V}_{\pUPs} = \{[\widetilde A]\in\hat V_{\URs}\,|\,\supp(A)\subseteq \pUP\}$ is a $\C\widetilde U$-submodule of $\hat V_{\pUPs}$ which is isomorphic to $\hat V_{\pUPs}$ under the restriction map above. Thus we may identify the $\C\widetilde U$-modules $\hat V_{\pUPs}$ and $\hat{\widetilde V}_{\pUPs}$.  
\hfill$\square$
\end{Remark}

\begin{Remark}\label{4action}
In illustration to come we picture linear characters $[A]\in\hat V_{\pUPs}$ for $A\in V$ as triangular shaped arrays of elements of $\F_q$, omitting from matrix $A$ all entries zero at positions not in $\pUP$. Moreover, indices placed just below the diagonal  denote both, row and column index of the corresponding matrix $A$.
By lemma \ref{rootsubgroupsU} the elements of $X_{ij}$ with $(i,j)\in \pUP$ can be written as products of elements of certain root subgroups $\tilde X_{st}\in \widetilde U$ with $(s,t)\in \tilde \Phi^+$. Combining this with
\ref{TruncatedColumnOperationG} we can illustrate below  the ``$.$''-action of $x_{ij}(\alpha)$ for $\alpha\in\F_q$ on $\hat V$ in theorem \ref{monomial}.  Recall  that $\tilde n=n+1$ for type $\fB_n$, and $\tilde n=n$ otherwise. Moreover if $1\leq i < j\leq N$, note that $\varepsilon_{ij}=-1$ for type $\fC_n$ and $j > n$ and $\varepsilon_{ij}=1$ otherwise (see \ref{description of R}).  If no ambiguity arises, we shall drop the indices and write $\varepsilon = \varepsilon_{ij}$. 
\end{Remark}

\vspace{4cm}
\begin{equation}\label{illTruncatedColumnOperationG}
\begin{picture}(200,200)

%  obere Pic 1)

\multiput(-20,203)(0,7){11}{\line(0,-1){2}}
\put(-20,275){\line(-1,0){80}}
\put(-20,275){\line(1,0){80}}
\put(-100,275){\line(1,-1){80}}
\put(60,275){\line(-1,-1){80}}
\put(-80,275){\line(0,-1){60}}
\multiput(-85,250)(0,-8){6}{\line(4,3){10}}
\put(-40,275){\line(0,-1){60}}
\put(0,275){\line(0,-1){60}}
\put(40,275){\line(0,-1){20}}
\put(-80,255){\circle*{4}}
\put(-91,253){\makebox{$i$}}
\put(-80,280){\makebox{$i$}}
\put(-51,210){\makebox{$j$}}
\put(-40,280){\makebox{$j$}}
\put(0,280){\makebox{$\bar j$}}
\put(40,280){\makebox{$\bar i$}}
\put(40,293){\line(0,1){10}}
\put(40,303){\line(-1,0){30}}
\put(10,303){\vector(-1,-2){5}}
\put(19,308){\makebox{$\alpha$}}
\put(-40,293){\line(0,1){10}}
\put(-40,303){\line(-1,0){30}}
\put(-70,303){\vector(-1,-2){5}}
\put(-67,308){\makebox{$-\alpha$}}
\put(-40,215){\circle*{4}}
\put(-130,200){\makebox{set to zero}}
\put(-105,213){\vector(2,1){12}}
\put(-39,253){\makebox{$(i,j)$}}
\put(-40,255){\circle*{4}}
\put(-70,180){\makebox{Illustration 1): $j\leq n$}}

%  obere Pic 2)

\multiput(200,203)(0,7){11}{\line(0,-1){2}}
\put(200,275){\line(-1,0){80}}
\put(200,275){\line(1,0){80}}
\put(120,275){\line(1,-1){80}}
\put(280,275){\line(-1,-1){80}}
\put(140,275){\line(0,-1){60}}
\multiput(135,250)(0,-8){6}{\line(4,3){10}}
\put(180,275){\line(0,-1){60}}
\put(220,275){\line(0,-1){60}}
\put(260,275){\line(0,-1){20}}
\put(181,253){\makebox{$(i,\bar j)$}}
\put(140,255){\circle*{4}}
\put(181,215){\circle*{4}}
\put(181,255){\circle*{4}}
\put(129,253){\makebox{$i$}}
\put(140,280){\makebox{$i$}}
\put(169,210){\makebox{$\bar j$}}
\put(180,280){\makebox{$\bar j$}}
\put(220,280){\makebox{$j$}}
\put(260,278){\makebox{$\bar i$}}
\put(260,291){\line(0,1){12}}
\put(260,303){\line(-1,0){70}}
\put(190,303){\vector(-1,-2){5}}
\put(233,306){\makebox{$\varepsilon\alpha$}}
\put(220,291){\line(0,1){17}}
\put(221,308){\line(-1,0){70}}
\put(150,308){\vector(-1,-2){8}}
\put(163,311){\makebox{$-\alpha$}}
\put(150,180){\makebox{Illustration 2): $\tilde n< j<\bar i$}}
%\end{picture}
%\end{equation*}
%\begin{equation}
%\begin{picture}(200,80)
\put(-20,95){\line(-1,0){80}}
\put(-20,95){\line(1,0){80}}
\put(-100,95){\line(1,-1){80}}
\put(60,95){\line(-1,-1){80}}
\put(-80,95){\line(0,-1){80}}
\multiput(-85,70)(0,-8){8}{\line(4,3){10}}
\put(-20,95){\line(0,-1){80}}
\put(40,95){\line(0,-1){20}}
\put(-80,75){\circle*{4}}
\put(-91,73){\makebox{$i$}}
\put(-80,100){\makebox{$i$}}
\put(40,100){\makebox{$\bar i$}}
\put(38,113){\line(-1,1){10}}
\put(28,123){\line(-1,0){38}}
\put(-10,123){\vector(-1,-2){7}}
\put(-33,100){\makebox{$n+1$}}
\put(10,128){\makebox{$\alpha$}}
\put(10,110){\makebox{\textcircled 2}}
\put(-24,111){\line(-1,1){12}}
\put(-36,123){\line(-1,0){34}}
\put(-70,123){\vector(-1,-2){7}}
\put(-65,128){\makebox{$-\alpha$}}
\put(-58,110){\makebox{\textcircled 1}}
\put(38,113){\line(-1,3){10}}
\put(29,143){\line(-1,0){97}}
\put(-68,143){\vector(-1,-3){11}}
\put(-32,148){\makebox{$\frac{1}{2}\alpha^2$}}
\put(-30,130){\makebox{\textcircled 3}}
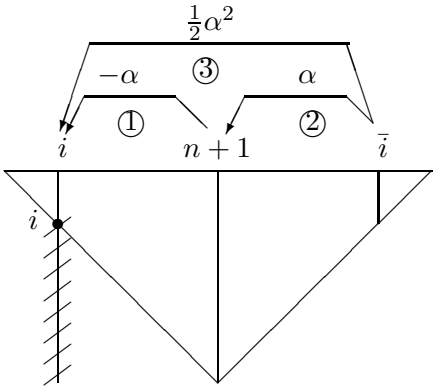
\put(-120,0){\makebox{Illustration 3): $j= n+1$ for type $\fB_n$}}
\multiput(200,3)(0,7){11}{\line(0,-1){2}}
\put(200,95){\line(-1,0){80}}
\put(200,95){\line(1,0){80}}
\put(120,95){\line(1,-1){80}}
\put(280,95){\line(-1,-1){80}}
\put(140,95){\line(0,-1){20}}
\put(135,71){\line(4,3){10}}

% Untere Pic 4

\put(260,95){\line(0,-1){20}}
\put(140,75){\circle*{4}}
\put(129,73){\makebox{$i$}}
\put(140,100){\makebox{$i$}}
\put(260,98){\makebox{$\bar i$}}
\put(260,75){\circle*{4}}
\put(275,58){\makebox{$(i,\bar i) \in \pUP$}}
\put(275,62){\vector(-1,1){10}}
\put(260,111){\line(-1,1){12}}
\put(248,123){\line(-1,0){93}}
\put(155,123){\vector(-1,-1){12}}
\put(183,131){\makebox{$-\alpha$}}
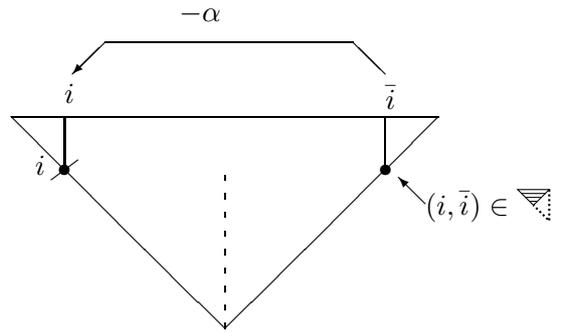
\put(117,0){\makebox{Illustration 4): $ j = \bar i$ for type $\fC_n$}}
\end{picture}
\end{equation}

%\vspace{1cm}

All the illustrations above  are self-explaining, in view of lemma \ref{rootsubgroupsU} and proposition \ref{TruncatedColumnOperationG}, except maybe no. 3). Here we need to stick to an order as given by 
\[
x_{i, n+1}(\alpha)=\tilde x_{i,n+1}(\alpha)\tilde x_{n+1,\, \bar i}(-\alpha)\tilde x_{i \bar i} (-\frac{1}{2}\alpha^2),
\]
labeled by \textcircled{1}, \textcircled{2} and \textcircled{3} in the illustration.

\begin{Remark}\label{actbyseq}
For $A\in V$ let $
\cO_A=\{[A].u\,|\,u\in U\}
$
be the {\bf $U$-orbit}
 of $[A]\in \hat V$ under the ``.''-action of $U$.
Note that in view of  \ref{rootsubgroupsU},  every $[B]\in \cO_{_A}$ ($A\in V$) can be obtained from $[A]$ by a sequence of restricted column operation in \ref{4action}.\hfill$\square$
\end{Remark}

\section{Staircase orbits}

 To determine the precise decomposition of $\C \hat V\cong \C U_{\C U}$ into orbit modules $\C \cO_{_A}$ we would need to classify the orbits $\cO_{_A}$, $A\in V$, that is e.g. by  finding a special set $\cR \subseteq \hat V$ of characters such that each orbit on $\hat V$ contains precisely one $[A]\in \cR$. To find such a collection $\cR$ will be subject to a future investigation. Instead we shall first exhibit a special subset  of all $U$-orbits on $\hat V$, called {staircase orbits}, and prove that each orbit module is isomorphic to a staircase orbit module.  In a second step we shall classify then the staircase orbits by so called core characters by proving, that each staircase orbit  contains precisely one core character. We first define staircase orbits.
 
  \begin{Notation}We set $\pUP=\tril\quad\dot{\cup}\quad \trir$,
where 
 \[
 \tril=\{(i,j)\in \pUP\,|\,j\leqslant \tilde n\}\quad \text{ and } \quad
 \trir=\{(i,j)\in \pUP\,|\,j>\tilde n\}
 \]
Note that  $\CC\subseteq \trir$ for type $\fC_n$. 
 \hfill $\square$
 \end{Notation}

\begin{Defn}
Suppose $A\in V$. We call $(i,j) \in \pUP$ a {\bfseries main condition} of $A$ (or of $[A]$) if $A_{ij}$ is the rightmost non-zero entry in the $i$-th row. We call a main condition $(i,j)$ {\bfseries left main condition} if $(i,j) \in \tril$,  and {\bfseries right main condition} if $(i,j)\in \trir$. Let
\begin{eqnarray*} 
\mc(A) & =& \big\{ (i,j) \in \pUP \; \big| \; \text{$(i,j)$ is a main condition of $A$} \big\}, \\
\lmc(A) & =& \big\{ (i,j) \in \pUP \; \big| \; \text{$(i,j)$ is a left main condition of $A$} \big\} \subseteq \tril, \\
\rmc(A) &= &\big\{ (i,j) \in \pUP \; \big| \; \text{$(i,j)$ is a right main condition of $A$}\big\}\subseteq \trir .
\end{eqnarray*} 
Note that we have $\mc(A)=\lmc(A)\;\dot\cup\; \rmc(A)$, and that  in view of remark \ref{actbyseq} $\mc(A)=\mc(B)$ for all $[B]\in \cO_{_A}$.\hfill$\square$
\end{Defn}
 
\begin{Defn}
Let $[A]\in \hat V$.  We call $[A]$ a {\bfseries staircase character},  if the elements of $\mc(A)$  lie in different columns and adopt a similar notation for $U$- and $\widetilde U$-orbits $\cal{O}$, and for $U$- and $\widetilde U$-orbit modules $M$.
\end{Defn}

To show that every $U$-orbit module is isomorphic to a staircase one we need to investigate homomorphisms between orbit modules. 
By general theory, every $U$-linear map from any right ideal $I$ of $\C U$ into $\C U$ is obtained by left multiplication $\lambda_x: I\longrightarrow \C U: y\mapsto xy$ by some $x\in \C U$, since $\C U$ is a self-injective algebra. We may use the right $\C U$-module isomorphism $f: \C U\longrightarrow \C \hat V\cong \C^V$ and its inverse $f^{-1}=f^*$ of \ref{monomial}  to induce  an $U$-linear left action $\lambda_x$ for $x\in U$ on $\C \hat V$ through multiplication of $x$ on $\C U$. We obtain:
\begin{Lemma}\label{leftactionall}
Let $x\in U, A\in V$. Then, identifying  $\C U_{\C U}$ and $\C \hat V\cong \C^V$ by $f$ and $f^*= f^{-1}$ we have:
\begin{equation}\label{leftmult1}
\lambda_x [A]=\sum\nolimits_{u\in U}{ \theta\circ \kappa(-x^{-t}A, u)}\pi(u).
\end{equation}
\end{Lemma}
\begin{proof}

Using $[A] = \chi_{-A}$ and \ref{monomial} we have:
\begin{eqnarray*}
\lambda_x f^*([A])&=&\lambda_x \sum_{u\in U}\theta\circ\kappa(-A,u)u
= \sum_{u\in U}{ \theta\circ \kappa(-A, u)}x u\\
&=& \sum_{u\in U}{ \theta\circ \kappa(-A, x^{-1}u)}u=\sum_{u\in U}{ \theta\circ \kappa(-x^{-t}A, u)}u,
\end{eqnarray*}
and the claim follows by applying $f$ on both sides of the formula above.
\end{proof}

\begin{Remark}\label{LeftPi}
Note that in the bilinear form of  \ref{leftmult1} we cannot replace  $u$ by $\pi(u)$, since the intersection of the support of $-x^{-t}A$ and of $u$ is not contained in $\pUP$ in general. In fact, the projection $\pi = \pi_{\pUPs}$ is not a left $1$-cocycle on $\widetilde U$, not even on $U$ (but on $U_{\pUPs}$, as we shall see below). The left action by $\lambda_x$ does not take $[A]$ to a multiple of $x.[A]:=[\pi(x^{-t}A)]$ in general, but into a linear combination of many characters $[C]\in \hat V$. \hfill$\square$
\end{Remark}   

\begin{Notation}
Let $\KL=\{(i,j)\in \Square\,|\, 1\leqslant j <\bar i \}$. Thus $\KL$ consists of all positions in $\Square$  above the anti-diagonal.  Set $\pKL=\KL$ for types $\fB_n$ and $\fD_n$,  and $\pKL=\KL\cup \CC$ for type $\fC_n$. \hfill$\square$
\end{Notation}

\begin{Theorem}\label{leftactionSupp}
Let $x\in U, A\in V$ such that $\supp(x^{-t}A)\subseteq \pKL$. Then $\supp(x^{-t}B)\subseteq \pKL$  for all  $[B]\in \cO_A$, and 
$
\lambda_x [B]=\theta\circ\kappa(-B,x^{-1})[\pi(x^{-t}B)] = \chi_{_{-B}}(\pi(x^{-1}))[\pi(x^{-t}B)]$.
\end{Theorem}
\begin{proof} Let $[B] \in \cO_A$. It is easy to see, that $x^{-t}B$ arises from $B$ by  adding scalar multiples of rows to lower rows. Now $[B] \in \cO_A$ implies, that, the most right hand sided non zero entries in rows of $B$  are on the main conditions with $B_{ij}=A_{ij}$ for all $(i,j)\in \main(A)=\main(B)$.  Hence  if $\supp(x^{-t}A)\subseteq \pKL$, then $\supp(x^{-t}B)\subseteq \pKL$ for all $[B] \in \cO_A$. By \ref{leftactionall}, we have:
\begin{eqnarray}
\lambda_x [B]&=&\sum\nolimits_{u\in U}{ \theta\circ \kappa(-x^{-t}B, u)}\pi(u)=\sum\nolimits_{u\in U}{ \theta\circ \kappa(-x^{-t}B, u-1+1)}\pi(u)
\nonumber\\&=&\theta\circ\kappa(-x^{-t}B,1)\cdot\sum\nolimits_{u\in U}{ \theta\circ \kappa(-x^{-t}B, u-1)\pi(u)}
\nonumber\\&=&\theta\circ\kappa(-B,x^{-1})\cdot\sum\nolimits_{u\in U}{ \theta\circ \kappa(-x^{-t}B, u-1)\pi(u)}
\nonumber\\&=&\chi_{_{-B}}(\pi(x^{-1}))\cdot\sum\nolimits_{u\in U}{ \theta\circ \kappa\big(-\pi(x^{-t}B), \pi(u)\big)\pi(u)} \label{leftactionsupp1}
\end{eqnarray}
since $\supp(x^{-t}B)\cap\supp(u-1)\subseteq \pUP$ and $\pi(u)=\pi(u-1)$.
Then the statement holds observing the following equation:
\begin{equation*}
\sum_{u\in U}\theta\circ \kappa\big(-\pi(x^{-t}B), \pi(u)\big)\pi(u)=\sum_{\pi(u)\in V}{ \chi_{-\pi(x^{-t}B)}\big(\pi(u)\big)\pi(u)}=[\pi(x^{-t}B)].
\end{equation*}
\end{proof}

Observe that $V=V_{\pUPs}$ is invariant under left multiplication by elements of $\widetilde U$. Hence $\widetilde U$ acts on $\C^V$. Indeed it permutes $\hat V \subseteq \C^V$. More precisely $u.\tau$ with $ u \in \widetilde U, \tau\in\C^V$ acts on $V$ by $u.\tau(B) = \tau (u^{-1}B)$ for $B\in V$. For the character $[A]\in\hat V, A\in V$, evaluation of  $u.[A]$ at $B\in V$ proves immediately: 

\begin{Lemma}\label{TypeAleftaction} Let $A\in V$ and $u\in \widetilde U$. Then $u.[A] = [\pi(u^{-t}A)]$.\hfill$\square$
\end{Lemma}   

Let $J= \tilde{\Phi}^+\setminus \pUP$. Then we have: 

\begin{Cor}\label{RestrRowTildeU} Let $1\leq i<j\leq N$ and $\alpha\in\F_q$. Then the permutation action of $\tilde x_{ij}(\alpha)$ on $\hat V$ is given by
\begin{equation}
\tilde x_{ij}(\alpha).[A]=[\pi(\tilde x_{ji}(-\alpha)A)] = [B]
\end{equation}
 where $B$ is obtained from $A$ by a {\bf restricted row operation}  which adds $-\alpha$ times row $i$ to row $j$ and projects the resulting matrix into $V$. In particular, if $\bar i<j$ then $[A] = [B]$, that is $\widetilde X_{ij}$ and hence $\widetilde U_J$ act trivially on $\hat V$. The left permutation action of $\widetilde U = \widetilde U_{\pUPs}\widetilde U_J$ on $\hat V$ is completely determined by the action of $\widetilde U_{\pUPs}$ and is generated by restricted row operations. 
  \hfill$\square$
  \end{Cor}

Observe that the restriction of the projection map $\pi = \pi_{\pUPt}: V_{\square}\rightarrow V$ to $\widetilde U_{\pUPs}$ is given as $\pi|_{\widetilde U_{\pUPt}}: \widetilde U_{\pUPs}\rightarrow V: u\rightarrow u-1$. Obviously $g\pi (u)+\pi(g)=g(u-1)+(g-1)=gu-1=\pi(gu)$ for all $g, u\in \widetilde U_{\pUPs}$ and hence $\pi|_{\widetilde U_{\pUPt}}$ is a left 1-cocycle.

The left hand sided version of theorem \ref{monomial} yields a left monomial $\widetilde U_{\pUPs}$-action on  $\hat V$ in addition to the already established right action. Indeed $V$ is the additive group of the Lie algebra of $\widetilde U_{\pUPs}$ and $\C \hat V = \C^V$ is a monomial $\C\widetilde U_{\pUPs}$-bimodule isomorphic to the regular  $\C\widetilde U_{\pUPs}$-bimodule. This is in fact Yan's original construction [\cite{yan}]. The action of $g \in \widetilde U_{\pUPs}$ on $\hat V$ is given as

\begin{equation}\label{pUPleftAction}
g[A] = \chi_{-A}(\pi(g^{-1}))[\pi(g^{-t} A)] = \theta\circ\kappa(-A,g^{-1})[\pi(g^{-t}A)].
\end{equation}
Thus for  the permutation action underlying the  monomial left action of $\widetilde U_{\pUPs}$ on $\hat V$ is the one given in \ref{TypeAleftaction}. 

 In general, the monomial left $\widetilde U_{\pUPs}$-action on $\hat V = \hat V_{\pUPs}$ does not commute with the monomial right $U$-action, but there are special cases, where this holds.
Let $x\in U$, then we find  a uniquely determined  element of $\widetilde U_{\pUPs}$, henceforth denoted by  $\tilde x$, such that $x=\tilde xz$ for some (as well uniquely determined) $z\in\widetilde U_J$ (see \ref{complements}).  

\begin{Lemma}\label{LeftTildeUandU} Let $x\in U$ and $[A]\in \hat V$. Then $x.[A] = \tilde x.[A]$. If in addition $\supp(x^{-t}A)\subseteq \pKL$, then 
\begin{equation}\label{UleftAction}
\lambda_x[A] = \theta\circ\kappa(-A,x^{-1})[\pi(x^{-t}A)] = \tilde x[A]
\end{equation}
 \end{Lemma}
\begin{proof} The left hand sided  equation of \ref{UleftAction} holds by \ref{leftactionSupp}. From \ref{RestrRowTildeU} we see that $x.[A] = \tilde x.(z.[A]) = \tilde x.[A]$. It remains to check the coefficient in \ref{UleftAction}. Now  \ref{RestrRowTildeU} implies in particular $\pi(z^{-t}A) = A$ and hence
$$
\theta\circ\kappa(-A,x^{-1}) = \theta\circ\kappa(-A,z^{-1}\tilde x^{-1}) = \theta\circ\kappa(-z^{-t}A,\tilde x^{-1}) = \theta\circ\kappa(-A,\tilde x^{-1})
$$ 
and the lemma follows.
\end{proof}

If $x\in U$ and $A\in V$ we write from now on $x[A]$ instead of $\lambda_x[A]$, provided $\supp(x^{-t}A)\subseteq \pKL$.

\begin{Cor}\label{lrcommute}
Let $A\in V=V_{\pUPs}$ and $g \in U$ (respectively $g\in \widetilde  U_{\pUPs}$)  such that $\supp(g^{-t}A)\subseteq  \pKL$. Then for all $u\in U$ and all $[B]\in \cO_{A}$ we have 
\[
(g[B])u=g([B]u),
\]
hence the monomial left action by $g$ on $\C \cO_{_A}$ commutes with the monomial right action by $U$. In particular the left operation of $g$ on $\cO_A$ induces an $U$-isomorphism  from $\C \cO_A$ onto $\C \cO_B$  where $[B]=g.[A]$.  \hfill$\square$
\end{Cor}

 \begin{Theorem}\label{isotostair}
 Let $A\in V$. Then there exists $B\in V$ such that $[B] \in \hat V$ is staircase and $\C \cO_A\cong \C \cO_B$ as $\C U$-module.
 \end{Theorem}
 \begin{proof}
 Suppose $\mc(A)$ contains two positions $(i,k), (j,k)$ in column $k$ with $1\leqslant k\leqslant N, 1\leqslant i< j \leqslant n$. By definition of $\mc(A)$,  we have $A_{il}=0$ for $l> k$ and hence  $\supp(\tilde x_{ji}(-\alpha)A)\subseteq \pKL$.

Then $\tilde x_{ij}(\beta).[A]=[B]$ for  $\beta = A_{jk}/A_{ik}\in\F_q$, where
 the main conditions of $A$ and $B$ coincide in all rows except the $j$-th one; Here the main condition $(j,k)$ of $A$ is deleted and possibly replaced  by a new main condition $(j,l)$ to the left of $(j,k)$, with $j< l< k$.
By corollary \ref{lrcommute}, $\C \cO_A$ and $\C \cO_B$ are isomorphic as $U$-modules. Since new main conditions in this procedure appear only to the left of the deleted ones, one may repeat the process working through the columns from right to left to produce finally a (not necessarily unique) staircase character $[C]\in \hat V$ such that $\C \cO_A\cong \C\cO_C$.
 \end{proof}
 
 \section{Core characters}
Our theorem \ref{isotostair} tells us that for finding the isomorphism classes of orbit module $\C \cO_A$, $A\in V$, it suffices to classify the staircase orbits. This will be done in the next two sections. Our strategy consists of using restricted column operations to annihilate as many nonzero entries as possible. So we strip $[A]$  all the way down to the assembly of positions with possibly non vanishing entries, called core of $A$, which we will define now:

\begin{Defn}\label{coredef} Let $[A]\in \hat V$ be a staircase character. 
\begin{itemize}
\item [1)]  The positions $(i, \bar j)$ with $(i,j)\in \rmc(A)$ are called {\bf minor conditions} of $A$ (or $[A]$, or $\cO_A$) and we denote the set of minor conditions of $[A]$ by $\minc(A)$ (or $ \minc(\cO_A)$). 
\item [2)] A position $(i,j)\in \tril$ is  called {\bf supplementary condition} for $A$ or $\cO_{_A}$,  if $(i,j)$ is on the left of some minor condition or some left main condition of $A$, in the same column as some minor condition of $A$ and is not itself a minor or main condition.  The set of supplementary conditions is denoted by $\suppl (A)$. Note that $\suppl(A)\subseteq \tril\setminus\{\text{column $\tilde n$}\}$.
\item[3)] The {\bf core} of $A$ or $\cO_A$ is defined to be 
$\core( A)=\main(A)\cup \minc(A)\cup \suppl(A)$ if  $U$ is of type $\fB_n$ or $\fD_n$ and $\core(A) =
\main(A)\cup \suppl(A)$ if  $U$ is of type $\fC_n$. Note that $\core( A)$ is determined by  $\main(A)$ and $\core(B) = \core(A)$ for all $B\in\cO_A$.
\item [4)] We define the {\bf verge} of $A$ to be $\verge(A)=\sum_{(i,j)\in \main(A)} A_{ij}e_{ij}$.
Note that $\supp\big(\verge(A)\big)=\main(A)$. The linear character $[A]\in \hat V$ is called {\bf verge character}, if $A=\verge(A)$.
\item [5)] A linear character $[A]\in \hat V$ is called {\bf core character}, if $\supp(A)\subseteq \core (A)$.
\end{itemize}
Note, that for type $\fC_n$ minor condition associated with antidiagonal main conditions are located on the diagonal and hence are not contained in $\pUP$. \hfill$\square$
\end{Defn}

 We shall see, that every staircase orbit contains precisely one core character.
 
\begin{Defn}\label{lowerhook}
For $(i,j)\in \UP$ we define the {\bf arm} at $(i,j)$ to be $\cA(i,j)=\{(\bar j, a)\in \pUP\}$, and the {\bf leg} at $(i,j)$  to be $\cL(i,j)=\{(a,j)\in \pUP\,|\,a>i\}$. If $U$ is of type $\fC_n$ we define in addition the {\bf arm} at $(i,\bar i)$ to be $\cA(i,\bar i) = \{(i, a)\in\UP\} =  \{(i,a)\,|\,i< a < \bar i\}$ for $1\leq i\leq n$. (For an illustration for these objects see \ref{corrpic} part i) below).

Moreover we define for any $(a,b) \in \mc(A)$ the the {\bf reduced leg}  $\cL(a,b)^\circ$ to be obtained by removing from $\cL(a,b)$ all positions which are contained in an arm of $A$ attached to a right main condition. Then we define
\[
\cL = \bigcup\nolimits_{(i,j)\in \mc(A)} \cL(i,j),\quad\cL^\circ = \bigcup\nolimits_{(i,j)\in \mc(A)} \cL(i,j)^\circ\quad\text{and}\quad \cA = \bigcup\nolimits_{(i,j)\in \mc(A)}\cA(i,j),
\]
then $\Limb(A) = \cL\cup\cA = \cL^\circ\dot\cup\cA$, the second union being disjoint. \hfill$\square$  
\end{Defn}

\begin{Defn}\label{defofplace}
Let $[A]\in \hat V$ be a staircase character, and let $(i,j)\in \pUP$. Then $(i,j)$ is called a {\bf place} of $A$ if $(i,j)\notin \core(A)$ and is to the left of a main condition. Thus $(i,j)$ is either
\begin{enumerate}
\item [\romannumeral1)] to the left of  and in the same row of a left main condition and not a supplementary condition or
\item [\romannumeral2)] to the left of and in the same row of a minor condition and not a supplementary condition or
\item [\romannumeral3)] between and in the same row of a minor and a main condition or
\item [\romannumeral4)] in the case  of type $\fC_n$,  a minor condition in $\pUP$. 
\end{enumerate}
 Note that the set of places of $A$ is determined by $\main(A)$ uniquely. It is hence denoted by $\Pl(A)$, or $\Pl(\cO_A)$, or
  $\Pl(\main(A))$. Note too that $\Pl(A)\subseteq \pUP.$\hfill$\square$
\end{Defn}

\begin{Remark}\label{threepositions} 
Let $[A]\in \hat V$ be a  staircase character.  Then the only nonzero entries of  any $[B]\in\cO_A$ are at positions $(i,j)\in\core(A)\cup \Pl(A)$, (disjoint union by definition of $\Pl(A)$).\hfill$\square$
\end{Remark}

Here is an illustration of the four types of places.  By $M$ we denote main, by $m$ minor and by $s$ supplementary conditions. The thick lines indicate the places attached to $M$ and $m$ respectively. The star dotted lines indicate the positions of $\Limb(A)$ connected to the places by the bijection explained below in \ref{placelimb} and dotted lines additional legs and arms.

% type i) being marked deep red, ii) dark brown iii) dark blue, and iv) circled blue. The light colored  positions (reduced legs and arms) are defined in %\ref{lowerhook} above. The light and the dark part to a particular color are connected by a bijection explained below in \ref{placelimb}

%\vspace{2.5cm}

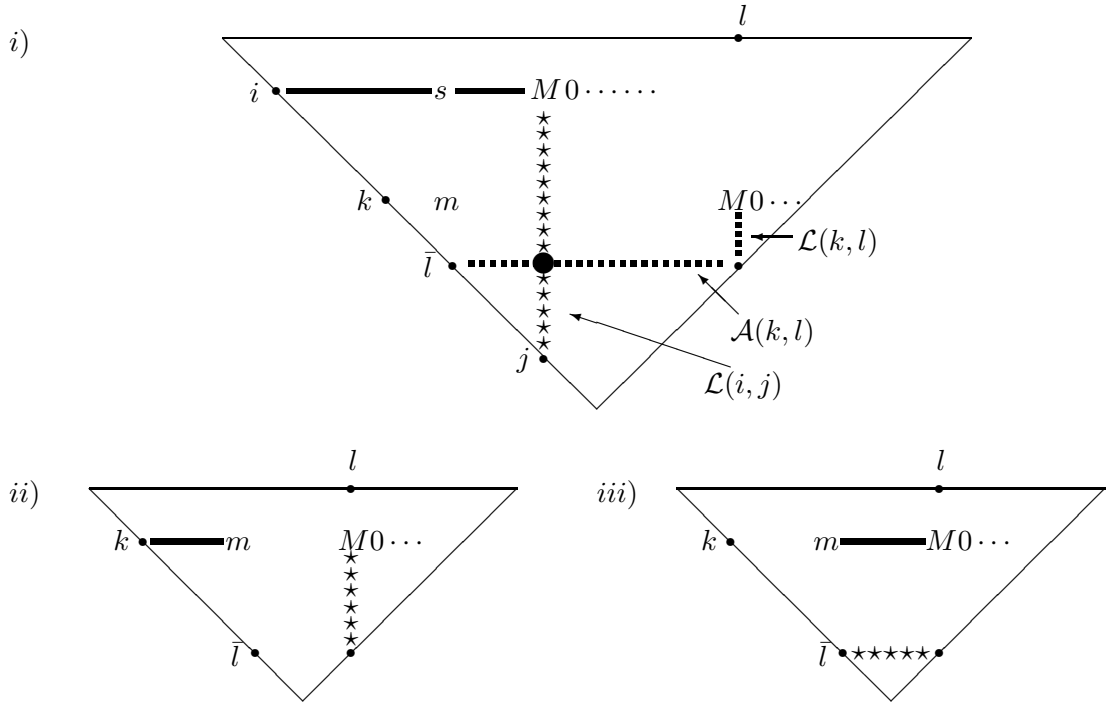
\begin{equation}\label{corrpic}
\begin{picture}(200,300)

%  obere Pic i)
\put(-130,270){\makebox{$i)$}}
%\linethickness{1pt}

\put(90,275){\line(-1,0){140}}
\put(90,275){\line(1,0){140}}
\put(-50,275){\line(1,-1){140}}
\put(230,275){\line(-1,-1){140}}

\put(70,190){\makebox{\circle*{8}}}
\put(-40,251){\makebox{$i$}}
\put(0,210){\makebox{$k$}}
\put(25,185){\makebox{$\bar l$}}
\put(143,280){\makebox{$l$}}
\put(60,151){\makebox{$j$}}
\put(-30,255){\circle*{3}}
\put(11,214){\circle*{3}}
\put(36,189){\circle*{3}}
\put(70,154){\circle*{3}}
\put(143,189){\circle*{3}}
\put(143,275){\circle*{3}} % l
\put(65,252){\makebox{$M$}}
\put(78,252){\makebox{$0\cdots\cdots$}}
\put(29,210){\makebox{$m$}}
\put(135,210){\makebox{$M$}}

\put(147,210){\makebox{$0\cdots$}}
\put(29,252){\makebox{$s$}}
%\thicklines
%places and corresponding Limb positions...
\linethickness{2pt}
\put(-26,255){\line(1,0){54}}
\put(37,255){\line(1,0){26}}
\dottedline[$\star$]{6}(70,245)(70,160)

%arms and legs
\linethickness{2pt}
\dottedline{4}(43,190)(136,190)
\dottedline{4}(143,208)(143,194)

\linethickness{.5pt}
\put(140,161){\makebox{$\cA(k,l)$}}
\put(143,171){\vector(-1,1){15}}

\put(130,141){\makebox{$\cL(i,j)$}}
\put(140,151){\vector(-3,1){60}}

\put(165,195){\makebox{$\cL(k,l)$}}
\put(163,200){\vector(-1,0){15}}

%\end{picture}
%\end{equation}

%\begin{equation*}
%\begin{picture}(200,155)

%  untere Pic ii) 

\put(-130,100){\makebox{$ii)$}}
\put(-20,105){\line(-1,0){80}}
\put(-20,105){\line(1,0){80}}
\put(-100,105){\line(1,-1){80}}
\put(60,105){\line(-1,-1){80}}
\put(-80,85){\circle*{3}} % k
\put(-91,82){\makebox{$k$}}
\put(-47,38){\makebox{$\bar l$}}

\put(-3,112){\makebox{$l$}}

\put(-49,82){\makebox{$m$}}
\put(-7,82){\makebox{$M$}}
\put(5,82){\makebox{$0\cdots$}}
\put(-38,43){\circle*{3}}  % \bar l
\put(-2,43){\circle*{3}}  %  l unten
%\put(-20,275){\circle*{4}} %Mitte
\put(-2,105){\circle*{3}}   % l

\linethickness{2pt}
\dottedline[$\star$]{6}(-2,79)(-2,48)
\put(-50,85){\line(-1,0){27}}
%\dottedline{5}(-4,213)(-4,275)

% Untere Pic iii)

\put(90,100){\makebox{$iii)$}}
\linethickness{.5pt}
\put(200,105){\line(-1,0){80}}
\put(200,105){\line(1,0){80}}
\put(120,105){\line(1,-1){80}}
\put(280,105){\line(-1,-1){80}}
\put(140,85){\circle*{3}} % k
\put(129,82){\makebox{$k$}}
\put(173,38){\makebox{$\bar l$}}

\put(217,112){\makebox{$l$}}

\put(171,82){\makebox{$m$}}
\put(213,82){\makebox{$M$}}
\put(225,82){\makebox{$0\cdots$}}
\put(182,43){\circle*{3}}  % \bar l
\put(218,43){\circle*{3}}  %  l unten
%\put(-20,275){\circle*{4}} %Mitte
\put(218,105){\circle*{3}}   % l

\linethickness{2pt}
\dottedline[$\star$]{6}(188,43)(212,43)
\put(181,85){\line(1,0){32}}

\put(-130,0){\makebox{iv) of \ref{defofplace} is obtained from i),ii) and iii) by including  minor conditions $m$.}}

\end{picture}
\end{equation}

%\begin{equation}\label{corrpic}
%\includegraphics[width=0.8\textwidth]{./Geo/correspondence}
%\end{equation}

\begin{Defn}\label{PlaceDisting} Let $(i,j)\in \mc(A)$. Then we denote by $\p(i,j)$ the set of places of $A$ to the left of $(i,j)$ in row $i$ if $(i,j)$ is a left main condition, and  to the left of the associated minor condition at position $(i, \bar j)$, if $(i,j)$ is a right main condition.
For $(k,l)\in \rmc(A)$ define in addition $\mathfrak m(k,l)$ to be the set of positions strictly between the right main condition $(k,l)$ and the associated minor condition $(k,\bar l)$ except in type $\fC_n$, where we include the minor condition itself  in case $k\neq \bar l$. 
Moreover define
\[
\mathfrak P=\bigcup_{(i,j)\in \mc(A)} \mathfrak p(i,j)\quad\text{and}\quad \mathfrak M= \bigcup_{(k,l)\in \rmc(A)}  \fm(k,l),
\]
so  $\Pl(\cO_A) = \mathfrak P\cup \mathfrak M$. \hfill$\square$
\end{Defn} 

We define a map $\f_A=\f:\Limb(A)\rightarrow \Pl(A)$ as follows:
For $(i,j)\in \lmc(A)$, we define $\f:\cL(i,j)\rightarrow \{(i,a)\in \UP\,|\, i<a<j\}$ setting $\f(a,j)=(i,a)$. This is obviously a bijection. Moreover, by \ref{coredef} part 2), inspecting figure \ref{corrpic} one sees immediately, that $\f$ maps positions  of the leg $\cL(i,j)$ intersecting some arm (as position $(\bar l,j)$ denoted by a big filled circle in \ref{corrpic})  bijectively to supplemental conditions in row $i$ to the left of $(i,j)$. Thus restricting $\f$ to $\cL(i,j)^\circ$ gives a bijection from  $\cL(i,j)^\circ$ onto  $\mathfrak p(i,j)$. Similarly for $(k,l)\in\rmc(A)$ setting $\f(b,l) = (k,b)$ we obtain a bijection  $\f$ from $\cL(i,j)^\circ$ onto $ \p(i,j)$. Finally, if $(k,l)\in \rmc(A)$  
we define $\f:\cA(k,l)\rightarrow \fm(k,l)$ by setting $\f(\bar l, a)=(k, \bar a)$ for $(\bar l, a)\in \cA(k,l) $ and again obtain a bijection. Thus:

\begin{Lemma}\label{placelimb}
Let $[A]\in \hat V$ be staircase. Then $\f_A$ is a bijection from $\Limb(A)$ to $\Pl(A)$.\hfill$\square$
\end{Lemma}

Given a staircase character $[A]$ we now want to construct a character $[B] \in \cO_A$  whose entries at places are prescribed values from $\F_q$. So  choose  arbitrary field elements $\lambda_{ij}\in \F_q$ for all $(i,j)\in \Pl(A)$. To adjust the entry at the $(i,j)$ to $\lambda_{ij}$  we shall act on the character by some particular root subgroup element $x_{kl}(\alpha)$, where $(k,l)\in\Limb(\cO_A)$ is uniquely determined by $\f(k,l) = (i,j)$. In addition we have to ensure, that the entries on places, which have been already dealt with, will not be changed any more by subsequent moves. This will be done, by choosing an order, in which we work through $\Limb(A)$, namely working first the reduced legs of $\Limb(A)$ from left to right, each top down. Then we use  the root subgroups on arms associated with right main conditions, ordering the latter from left to right as well as the positions on the arms itself. Note, that all this concerns only the permutation action underlying the monomial action of  $U$ on $\cO_A$, thus we have not to worry about coefficients. 
So let  $(t,s)$ be a main condition.  We assume that on positions $(i,a)\in \mathfrak{p}(i,j)$ for main conditions $(i,j)$ with $j < s$ and on positions $(t,b)\in\mathfrak{p}(t,s)$ with $t<b<a<s$ we have already established $B_{ia}=\lambda_{ia}$ and $B_{tb}=\lambda_{tb}$ for a $[B]\in\cO_A$.  We inspect $[C] = [B].x_{as}(\beta)\in\cO_A$. Note that $\f(a,s) = (t,a)$.  

Using \ref{illTruncatedColumnOperationG} one sees, that we obtain $C$ by adding $-\beta$ times column $s$  to column $a$ and $\varepsilon\beta$ times column $\bar a$ to column $\bar s$ of $B$, if $s\neq n+1$. If $s=n+1$, then $U$ is of type $\fB_n$ and $C$ is obtained from $B$ by adding $-\beta$ times column $n+1$ to column $a$, then $\beta$ times column $\bar a$ to column $n+1$ and finally $\frac{1}{2}\beta^2$ times column $\bar a$ to column $a$. Thus we have for $1\leq d < a$: 
\begin{equation}\label{fillbyleftlegs}
C_{da} = B_{da} - \beta B_{ds} \quad\text{and} \quad C_{d\bar s} = B_{d\bar s} + \varepsilon\beta B_{d\bar a}\quad \text{if } s\neq n+1,
\end{equation}
where $\varepsilon$ is defined as in \ref{4action}, and 
\begin{equation}\label{lm2}
C_{da}=B_{da}-\beta B_{ds}+\frac{1}{2}\beta^2 B_{d\bar a}\quad\text{and} \quad C_{ds} = B_{ds} + \beta B_{d\bar a}\quad\text{if } s=\bar s = n+1.
\end{equation}
In particular for $d=t$, since $(t,s)$ is a main condition, we have $B_{ts}\neq 0$, but $B_{t\bar a} = 0$ and hence we can choose $\beta\in\F_q$ such, that $C_{ta} = \lambda_{ta}$. Moreover, since the position $(t,\bar a)$ is to the right of the main condition $(t,s)$,  we have $C_{t\bar s} = B_{t\bar s}$ and 
$0 = B_{t\bar a}$.

Now let $d \neq t$. If $B_{ds}$ or $B_{d\bar a}$ is nonzero, there must be a main condition $(d,r)$ in row $d$ in column $r$ strictly to the right of column $s$. This can be a left or a right main condition and in the latter case the corresponding minor condition $(d,\bar r)$ can be to right or to the left of $(d,a)$. Thus $(d,a) \in \mathfrak{p}(d,r)$, or $(d,a) \in \mathfrak{m}(d,r)$ with $s< r$. However the entries of positions to the left of $(d,r)$ will be adjusted later.
Moreover, if $(i,j)$ is a main condition in an earlier column than $(t,s)$, that is if $j < s$, we assume that $B_{ia}=\lambda_{ia}$ is already adjusted. But then, for  $d=i$  we see inspecting \ref{fillbyleftlegs}, that $C_{ia} = B_{ia}$, since $B_{is}=0$, the position $(i,s)$ being to the right of the main condition  $(i,j)$. Thus entries on places $\mathfrak{p}(i,j)$ for earlier main conditions $(i,j)$ remain unchanged under the action of $x_{as}(\beta)$.

\medskip
Here is an illustration for the case $s\leq n$ (the case $s>n$ being similar):

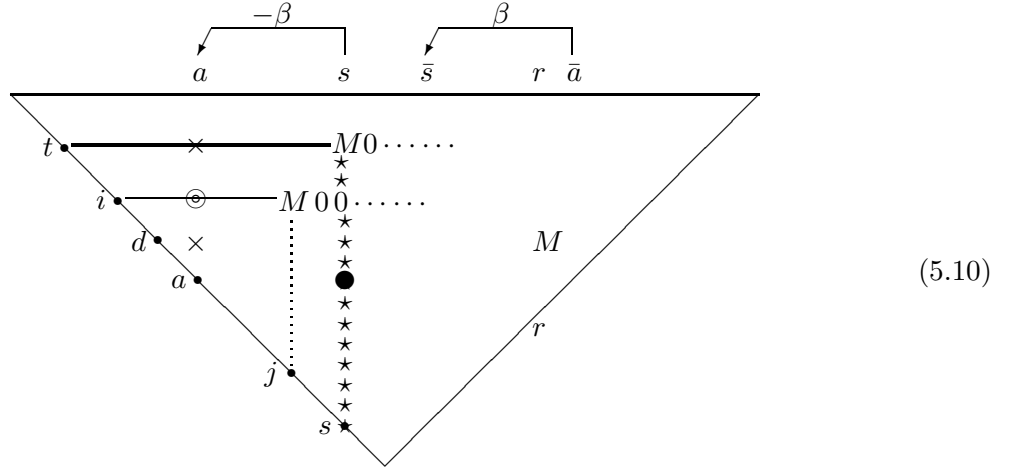
\begin{equation}\label{leftmainplaces1}
\begin{picture}(122,100)

% Triangle

\put(60,70){\line(-1,0){140}}
\put(60,70){\line(1,0){140}}
\put(-80,70){\line(1,-1){140}}
\put(200,70){\line(-1,-1){140}}

% dots denoting positions

\put(-60,50){\circle*{3}}
\put(-68,47){\makebox{$t$}}
\put(-40,30){\circle*{3}}
\put(-48,27){\makebox{$i$}}
\put(-25,15){\circle*{3}}
\put(-35,11){\makebox{$d$}}
\put(115,11){\makebox{$M$}}
\put(-10,0){\circle*{3}}
\put(-20,-3){\makebox{$a$}}

%Hilfslinien
%\put(8,-18){\line(1,0){140}}
%\put(112,-18){\line(0,1){140}}
%\put(-11,-0){\line(0,1){140}}
%\put(115,-21){\line(-1,0){140}}
%\put(-11,0){\line(0,1){140}}

\put(25,-35){\circle*{3}}
\put(15,-38){\makebox{$j$}}
\put(45,-55){\circle*{3}}
\put(35,-58){\makebox{$s$}}

% top labels

\put(-12,75){\makebox{$a$}}
\put(42,75){\makebox{$s$}}
\put(128,75){\makebox{$\bar a$}}
\put(115,75){\makebox{$r$}}
\put(115,-21){\makebox{$r$}}
\put(73,75){\makebox{$\bar s$}}

%places and flip of places 
\linethickness{1pt}
\put(40,48){\makebox{$M 0\cdots\cdots$}}  %(t,s)
\dottedline{1}(-57,51)(39,51)
\dottedline[$\star$]{8}(45,-55)(45,22)

% lonely stars
\put(41,42){\makebox{$\star$}}  
\put(41,35){\makebox{$\star$}}

% Main condition links davon
\linethickness{.5pt}

\put(20,26){\makebox{$M\, 0\, 0\cdots\cdots$}}  %(i,j)
\dottedline{.5}(-37,31)(19,31)
\dottedline{3}(25,22)(25,-35)

% action (a,s) on place (t,a)
\put(45,0){\circle*{7}}
%\put(125,-45){\makebox{$(a,s)$}}
%\put(120,-38){\vector(-2,1){70}}
\put(-15,48){\makebox{$\times$}}
\put(-15,28){\makebox{$\circledcirc$}}
\put(-15,11){\makebox{$\times$}}

% action by $x_{(t,s)}(\beta}$:

\put(130,85){\line(0,1){10}}
\put(130,95){\line(-1,0){50}}
\put(80,95){\vector(-1,-2){5}}
\put(100,97){\makebox{$\beta$}}

\put(45,85){\line(0,1){10}}
\put(45,95){\line(-1,0){50}}
\put(-5,95){\vector(-1,-2){5}}
\put(10,97){\makebox{$-\beta$}}

 \end{picture}
\end{equation}

\vskip100pt
Note that the entry at position $(t,a)$ is changed by the action of $x_{(a,s)}(\beta)$, that at position $(d,a)$ is changed, if $B_{(ds)}\ne 0$, (both marked $\times$),  whereas the entry at position $(i,a)$ marked $\circledcirc$ is not changed in \ref{leftmainplaces1} above).

\begin{Lemma}\label{fLM}
Let $[A]\in \hat V$ be staircase. For each $(t,a)\in \mathfrak P$ choose a scalar $\lambda_{ta}\in \F_q$. Then there exist a unique  $y=\prod_{(a,s)\in \cL^\circ}x_{as}(\alpha_{as})$,  where $\alpha_{as}\in \F_q$ for $(a,s)\in \cL^\circ$, such that  $[B]=[A].y\in\cO_A$ satisfies $B_{ta}=\lambda_{ta}$ for all $ (t,a)\in \mathfrak P$.
\end{Lemma}
\begin{proof}
Note that by our construction $(t,a) = \mathfrak f(a,s)\in\mathfrak p(t,s)\subseteq \mathfrak P$.  Everything but the uniqueness part has been proved in our discussion above. Recall that we ordered $\cL^\circ$ by going through main conditions from left to right and each leg top down. This is obviously the order in which the product in $y=\prod_{(a,s)\in \cL^\circ}x_{as}(\alpha_{as})$ has to be taken. To determine the $\beta$ in \ref{fillbyleftlegs} (replacing $B$ by $A$ and $C$ by $B$) such that $B_{ts}=\lambda_{ts}$, amounts to solve a linear equation with a unique solution, and hence $y$ is unique.
\end{proof}

For later use we need the following auxiliary lemma:
 
\begin{Lemma}\label{fLMaux}
Let $[B]\in\hat V$ be staircase, $(t,s)\in\mc(B)$ and $(a,s)\in \cL^\circ(t,s)$. Thus $\f(a,s)=(t,a)\in\mathfrak p(t,s)$. Let $[C] = [B].x_{as}(\beta)$ with $\beta\in\F_q$. Assume in addition $B_{rw}=0$ for all $(r,w)\in{\fM}$, then $C_{rw}=0$ for all $(r,w)\in \fM$ as well and $C_{cd}=B_{cd}$  for all $(c,d)\in \core(A)$.
\end{Lemma}
\begin{proof}
Recall that $B$ and $C$ differ only in columns $a$ and $\bar s$ (compare  \ref{leftmainplaces1}). If $(r,a)\in\mathfrak M$ then we have a minor condition $(r,c)$ with $c<a<s<\bar a<\bar c$ and $c<\bar s$ and hence $(r,a), (r,s), (r,\bar s), (r,\bar a)\in\mathfrak m(r,\bar c)$. With \ref{fillbyleftlegs} and \ref{lm2} we conclude $B_{r,a}=C_{r,a}=0$. 
Moreover $B$ and $C$ coincide on main conditions because they are elements of the same $U$-orbit of $\hat V$. Now $(a,s)\in \cL^\circ(t,s)$ implies that $(t,a)\not\in\minc(B)\cup\suppl(B)$, and hence there is neither a minor nor a supplemental condition in column $a$ by \ref{coredef}. If $(t,s)\in\lmc(A)$, then $s= n+1$ or  column $\bar s\subseteq \trir$. In both cases column $\bar s$  does not contain any minor or supplementary condition. So let $(t,s)\in\rmc(A)$. As before we argue, that column $a$ does not contain any minor or supplementary condition. Moreover, for the minor condition $(t,\bar s)$ in column $\bar s$ we have $C_{t\bar s} = B_{t\bar s}$ by \ref{fillbyleftlegs} and \ref{lm2}, since $(t,\bar a)$ is to the right of the main condition $(t,s)$ and therefore $B_{t\bar a}=0$. If $(i,\bar s)\in\suppl(B)$ with $i\not = t$ is another supplemental condition in column $\bar s$,  there must be a minor or left main condition $(i,\bar j)$ in row $i$ to the right of $(i,\bar s)$. Thus $\bar s < \bar j$ and hence $\bar j <j<s<\bar a$. Since either $(i,j)$ or $(i,\bar j)$ is a main condition, we conclude $B_{i\bar a}=0$. Again by \ref{fillbyleftlegs} and \ref{lm2} we obtain $C_{i\bar s} = B_{i\bar s}$, as desired.
\end{proof}

 Recall that  in \ref{4action} we defined  for the action of $X_{ij}$ on $\hat V$ always $\varepsilon=1$ except for type $\fC_n$ and $j > n$, where we set $\varepsilon=-1$.

\begin{Lemma}\label{fM}
Let $[B]\in \hat V$ be staircase and let $(k,l)\in\rmc(B)$. Suppose $\bar l<a< l$ or $\bar l\leq a<l$, if $U$ is of type $\fC_n$ . Then $(\bar l,\bar a)\in\cA(k,l)$ and 
$[C] = [B].x_{\bar l\bar a}(\beta)$ with $\beta\in\F_q$  satisfies:
\begin{enumerate}
\item [(1)] 
$C_{ra} = B_{ra} + \varepsilon\beta B_{rl}$, for $1\leq r<\bar l$. Moreover $C_{r\bar l} = B_{r\bar l} - \beta B_{r\bar a}$ unless $U$ is of type $\fB_n$ and $a = \bar a = n+1$. In this case $C_{r\bar l} = B_{r\bar l} -\beta B_{r,a}+\frac{1}{2}\beta^2B_{rl}$.
\item [(2)] Let $(r,s)\in\mathfrak P$ or $(r,s) \in\mathfrak m(d,t)$ with $(d,t)\in\rmc(B)$ with $t<l$. Then $C_{rs} = B_{rs}$.
\end{enumerate}
\end{Lemma}

\begin{proof} 
Part (1) follows immediately from \ref{4action}.  
To prove (2)  observe that $C$ and $B$ differ only in columns $a$ and $\bar l$ and that we may  assume that $r\not=k$. Suppose $B_{rl}\not = 0$ for some $r\not=k$. Then  $(r,t)\in\rmc(B)$ for some $t$ with $l<t$, since $\tilde n < l$ and  $[B]$ is staircase. Thus $(r,\bar t)$ is a minor condition with $\bar t <\bar l <a$. This implies that $(r,a)\in\mathfrak m(r,t)$  with $l<t$ and hence  (2) holds for $s=a$.  The case $s=\bar l$ follows similarly.
\end{proof}

\begin{Cor}\label{fM2} Let $[B]\in\hat V$ be staircase. Choose $\lambda_{ka}\in\F_q$ for each $(k,a)\in\fM$ and let $(\bar l,\bar a)\in \cA$ be such, that $\mathfrak f(\bar l,\bar a) = (k,a)$. Then there exists a unique $y= \prod_{(\bar l,\bar a)\in \cA}x_{\bar l \bar a}(\alpha_{\bar l \bar a})$ with $\alpha_{\bar l \bar a}\in\F_q$ such that $[C] = [B].y$ satisfies $C_{ka}=\lambda_{ka}$ for all $(k,a)\in\fM$ and $C_{rs}=B_{rs}$ for all $(r,s)\in\mathfrak P$.
\end{Cor}
\begin{proof}
Obviously we can find a uniquely determined $\beta\in\F_q$ in \ref{fM} such that $C_{ka}=\lambda_{ka}$. Now the claim follows, applying \ref{fM} several times working through the arms associated with right main conditions in our previously defined order.   
\end{proof} 

Combining \ref{fLM} and \ref{fM2} the main result of this section follows:

\begin{Theorem}\label{Orbitdescription} Let $[A]\in \hat V$ be staircase. Choose $\lambda_{kl}\in \F_q$ for all $(k,l)\in \Pl(A)$. Then there exist uniquely determined $\alpha_{ij}\in \F_q$  for $(i,j)\in \Limb(A)$  such that $[B] = [A]. \prod_{(i,j)\in \Limb(A) }x_{ij}(\alpha_{ij})$ satisfies
$B_{kl}=\lambda_{kl}$ for all $(k,l)\in \Pl(A)$.\hfill$\square$
\end{Theorem}

As a further consequence we obtain a lower bound for the size of the orbit $\cO_A$ for $[A]\in \hat V$ being a staircase:
\begin{Cor}\label{orbitsize}
Let $a=|\Pl(A)|$, the number of places of linear characters in $\cO_A$. Then $a=|\Limb(A)|$ and $|\cO_A|\geqslant q^a$.\hfill$\square$
\end{Cor}

Choosing $\lambda_{kl}=0$ for all $(k,l)\in \Pl(A)$ we get as special case of \ref{Orbitdescription}:

\begin{Cor}\label{tocore}
Each staircase orbit contains a core character.\hfill$\square$
\end{Cor}

For later use, we state a technical lemma:
\begin{Lemma}\label{omit1move}
Let $[A]\in \hat V$ be staircase such that  $A_{rw}=0$ for all positions $(r,w)\in \fM$. Let $[C]\in \cO_A$ be the core character  derived by \ref{tocore}. Then 
\begin{equation}\label{omit1moveformula}
C=\sum_{(i,j)\in \core(A)} A_{ij}e_{ij}. 
\end{equation}
\end{Lemma}
\begin{proof}

By \ref{Orbitdescription} there exists a unique $y$ such that
$[C]=[A].y$ with  $y$  written as $y=y_{_1}y_{_2}$ where
\[
y_{_1}= \prod_{(i,j)\in \cL^\circ}x_{ij}(\alpha_{ij});   \quad y_{_2}= \prod_{(i,j)\in  \cA}x_{ij}(\alpha_{ij})
\]
 with uniquely determined $\alpha_{ij}\in \F_q \text{ for } (i,j)\in \Limb(A)$.
 Let $[B]=A. y_{_1}$. Then by \ref{fM2}, we have $[C]_{ab}=[B]_{ab}$ for all $(a,b)\in \mathfrak P$. Since $[C]$ is a core, we have
 $[B]_{ab}=[C]_{ab}=0$ for all $(a,b)\in \mathfrak P$. Moreover by \ref{fLMaux}  we have  $[B]_{ab}=[C]_{ab}=0$ for all $(a,b)\in  \fM$, since $[A]$ satisfies assumption: $A_{rw}=0$ for all positions $(r,w)\in \fM$. Thus  $[B]_{ab}=[C]_{ab}=0$ for all $(a,b)\in \Pl(A)=\mathfrak P\cup\fM$. This implies 
   $[B]=[C]$ and $y_{_2}=1$. Then the formula  (\ref{omit1moveformula}) follows by  the statements of the values on positions $(c,d)\in \core(A)$ in \ref{fLMaux} and \ref{fLM}.
 \end{proof}

\begin{Remark}
By the lemma above, we can also derive a core character in a staircase orbit by first putting zeros in positions of $\fM$, then zeros in  positions of $\mathfrak P$. In the next section we shall show the core character is unique in a staircase orbit.\hfill$\square$
\end{Remark}

\section{Classification of staircase orbits}
In the last section we have seen that  every staircase orbit $\cO_A, A\in V$ contains a core character $[C]\in \cO_A$. Now we shall show that this core character is unique. Thus the core characters classify the staircase  $U$-orbits in $\hat V$.

\begin{Lemma}\label{def of J(A)}
Let $[A]\in\hat V$ be be a staircase character. Then $J(A)= \pUP \setminus \Limb(A)$ is a closed subset of $\pUP$. Therefore the subgroup $U_{J(A)}$ of $U$ generated by the root subgroups $X_{ij}$ with $(i,j)\in J(A)$ is a pattern subgroup of  $U$.
\end{Lemma}
\begin{proof}
Firstly let $(i,j), (j,k)\in \pUP$ and assume that $(i,k)\in \Limb(A)$. Suppose $(i,k)\in \cL(l,k)$ with $(l,k)\in \mc(A)$. Then $l<i$. Since $(i,j)\in \pUP$, we have $l<i<j$ and hence  $(j,k)\in \pUP$ implies $(j,k)\in \cL(l,k)$. Now suppose $(i,k)\in \cA(s, \bar i)$ with $(s, \bar i)\in \mc(A)$. Since $i<j<k$, we have $(i,j)\in \cA(s, \bar i)$. Thus $(i,j), (j,k)\notin \Limb(A)$ implies $(i,k)\notin \Limb(A)$.

Secondly suppose $(i,j), (\bar k, \bar j)\in \pUP$ and assume that $(i,k)\in \Limb(A)$. Suppose $(i,k)\in \cL(l,k)$ with $(l,k)\in \mc(A)$. Since $(\bar k, \bar j)\in \pUP$ we have $1\leqslant \bar k \leqslant n$ and hence $k>n$. Then $\cA(l,k)$ is well-defined and $(\bar k, \bar j)\in \cA(l,k)$ since $(\bar k, \bar j)\in \pUP$. Now suppose $(i,k)\in \cA(s, \bar i)$ with $(s, \bar i)\in \mc(A)$. Since $(i,j), (\bar k, \bar j)\in \pUP$ we have $i<j<k$ and hence $(i,j)\in  \cA(s, \bar i)$. Thus $(i,j), (\bar k, \bar j)\notin \Limb(A)$ implies $(i,k)\notin \Limb(A)$.
The lemma follows by \ref{Upattern} and  \ref{def of generator}.
\end{proof}

Now  corollary \ref{orbitsize} and lemmata \ref{placelimb} and \ref{def of J(A)} imply immediately:

\begin{Cor}\label{stabupbound} For $[A]\in\hat V$ define $\Stab_U[A]=\{u\in U\,|\,[A].u = [A]\}$. Let $J(A)= \pUP \setminus \Limb(A)$, $m=|J(A)|$ and $a=|\Limb (A)|$. Then $a=|\Pl(A)|$  and  $|\pUP| = m+a$. Moreover $q^a \leq |\cO_A| = [U:\Stab_U[A]]$ and hence $|\Stab_U[A]|\leq q^m = |U_{J(A)}|$. \hfill$\square$
\end{Cor}

\begin{Lemma}\label{core position invariant}
Let $[A]\in \hat V$ be a staircase character and  $J(A)$ be defined as in lemma \ref{def of J(A)}. If $\supp(A)\cap \trir=\rmc(A),$ then 
$
[B] = [A].x_{ij}(\alpha)=[A].\tilde x_{ij}(\alpha) \text{ for all } (i,j)\in J(A), \alpha\in \F_q.
$
Moreover for $u\in U_{J(A)}$ we have $
([A].u)_{st}=[A]_{st},\,\forall \, (s,t)\in \core(A).
$
\end{Lemma}
\begin{proof} 
If column $\bar i$ contains a nonzero value at position $(k,\bar i)\in \rmc$, position $(i,j)$ is on $\cA(k,\bar i)$ and hence not contained in $J(A)$. Thus $(i,j)\in J(A)$ implies that column $\bar i$ is a zero column and inspecting \ref{illTruncatedColumnOperationG} we see that the first claim of the lemma holds. Moreover $A$ and $B$ differ only in column $i$ and in particular  $\supp(B)\cap \trir=\rmc(A) = \rmc(B)$. 
By the same argument column $i$ cannot contain a minor condition and hence no supplementary conditions as well. Moreover, if it contains a left main condition $(l,i)$ then $A_{lj}=0$ and hence adding a multiple of column $j$ to column $i$ will not change $A_{li}$. Thus $(l,i)$ is  then the only core position on column $i$, and we conclude $B_{st}=A_{st}$ for all $(s,t)\in \core(A)$. From this the second claim of the lemma follows immediately.
\end{proof}

\begin{Cor}\label{J in stab}
Let $[A]\in \hat V$ be a verge character, and  $J(A)$ be defined as in lemma \ref{def of J(A)}.  Then  $U_{J(A)} =  \Stab_U[A]$.
\end{Cor}
\begin{proof} By \ref{stabupbound} it suffices to show $U_{J(A)}\subseteq \Stab_U[A]$. Let $(i,j)\in J(A)$ and $\alpha\in \F_q$ and let $[B]=[A].x_{ij}(\alpha)$.  Since $[A]$ is a verge, it is a core. By \ref{core position invariant} $[B]$ and $[A]$ differ at most in column $i$, and this happens only, if column $j$ of $A$ contains a nonzero entry above row $i$. Since $[A]$ is a verge, such a nonzero entry has to be a main condition, forcing $(i,j) \in \Limb(A)$, a contradiction. Thus $[B]=[A]$ and the assertion of the lemma follows.
\end{proof}

\begin{Lemma}\label{auxorderlemma}
Let $[A]$ be a core character and let $J(A)$ be defined as in lemma \ref{def of J(A)}. Then  $|U_{J(A)}| = |\Stab_U(A)|$.
\end{Lemma}
\begin{proof}
By \ref{stabupbound} we only need to show $|U_{J(A)}| \leq |\Stab_U(A)|$.
Let $u\in U_{J(A)}, [B]=[A].u$ and $[A_0]=\verge(A)$ and $A_1=A-A_0\in V$.  By \ref{J in stab}, we have $[A_0]=[A_0].u$, and hence $A_0=\pi(A_0 u^{-t})$ with $\pi = \pi_{\pUPt}$ as in \ref{UactsAuto}. Thus 
$B=A.u^{-t}=\pi(Au^{-t})=\pi((A_0+A_1)u^{-t})=A_0+A_1.u^{-t}$.
Since $[A]$ is a core character, we have $\main(A_1)\subseteq \minc(A)\cup \suppl(A)$, and hence $[B]$ can be different from $[A]$ only at positions in the row and to the left of minor or supplementary conditions.
 Therefore $[B]$ satisfies the assumption of corollary \ref{omit1move} and lemma \ref{core position invariant}. 
  By \ref{tocore} there exists $\lambda_{ij}\in \F_q$ for each  $(i,j)\in \Limb(A)$, such that $[B].y_u$ is a core character, and by \ref{Orbitdescription}, $y_u$ is uniquely determined, where 
$
y_u= \prod_{(i,j)\in \Limb(A) }x_{ij}(\lambda_{ij}),
$
the product is taken again in the order of $\Limb(A)$ defined in the previous section.
By  \ref{omit1move} 
\begin{equation}\label{Auy1}
[A].uy_u=[B].y_u\quad \text{with}\quad B=\sum_{(s,t)\in \core(B)}B_{st}e_{st}=\sum_{(s,t)\in \core(A)}B_{st}e_{st}.
\end{equation}
Moreover by \ref{core position invariant},  we have $[B]_{st}=([A].u)_{st}=[A]_{st}$ for all $(s,t)\in \core(A)$. Observing that $[A]$ is a core character,
 equation \ref{Auy1} implies $[A].uy_u=[A]$.
Thus $uy_u\in \Stab_{U}(A)$. Now we define a map:
\[
\phi: U_{J(A)}\rightarrow \Stab_{U}(A): u\mapsto u y_u.
\]
We prove that $\phi$ is injective, then we have shown 
$|U_{J(A)}|\leqslant |\Stab_U(A)|$. For $u, v\in U_{J(A)}$ we define the corresponding $y_u$ and $y_v$ as in \ref{tocore}. Assume $uy_u=vy_v$. Recall that $\pUP$ is the disjoint union $J(A)\dot\cup\Limb(A)$. Thus fixing  a linear ordering of $J(A)$, using the linear ordering of $\Limb(A)$ defined in the last section and defining the positions in $\Limb(A)$ to come after all elements of $J(A)$ defines a linear ordering of $\pUP$. The uniqueness part of \ref{anyorder} ensures  $u=v$. Therefore $\phi $ is injective as desired.
\end{proof}

\begin{Theorem}\label{classifybycore}
Let $[A]\in \hat V$ be staircase. Then for $\Lambda:\Pl(A)\longrightarrow \F_q: (i,j)\mapsto \lambda_{ij}$ there exists precisely one $[B]=[B(\Lambda)]\in \cO_A$ with $B_{ij}=\lambda_{ij}$ for all $(i,j)\in \Limb(A)$. Moreover
\[
\cO_{A}=\{[B(\Lambda)]\,|\,\Lambda:\Pl(A)\longrightarrow \F_q\}
\]
In particular, $\cO_A$ contains precisely one core character $[A_0]=[B(\Lambda_0)]$ with $\Lambda_0: \Pl(A)\longrightarrow \F_q: (i,j)\mapsto 0$ for all $(i,j)\in \Pl(A)$. As a consequence 
\[
|\cO_A|=q^{|\Pl(A)|}=q^{|\Limb(A)|}=q^{|\Limb(\verge(A))|}=[U:U_{J(A)}].
\]
\end{Theorem}
\begin{proof}
In \ref{fLM} and \ref{fM2} we constructed for each
 $\Lambda$ an element $[B]= [B(\Lambda)]\in \cO_A$. Thus $|\cO_A|\geqslant q^a$ with $a=|\Pl(A)|$ (comp \ref{orbitsize}). By \ref{placelimb} we have  $|\Pl(A)|=a=|\Limb(A)|$. 
  If $|J(A)|=m$, we have $|\pUP|=a+m$ and hence $q^{a}=[U:U_{J(A)}]$.
By \ref{auxorderlemma} we obtain 
$$q^a\leqslant |\cO_A|=[U:\Stab_U([A])]=[U:U_{J(A)}] =q^a,$$ and therefore we have equality. From this all claims of the theorem follow immediately.
\end{proof}

\section{Andr\'{e}-Neto supercharacters}\label{ANsuperchar}

In this last section we shall decompose the Andr\'{e}-Neto supercharacters of $U$ into characters afforded by orbit modules $\C\cO_A$, for certain $[A]\in \hat V$.  First we briefly describe the  Andr\'{e}-Neto elementary characters and relate those to our set up. Recall that $\tilde n=n+1$ for type $\fB_n$ and  $\tilde n=n$ otherwise.

\begin{Defn}\label{ANelemsubgps}
For $(i,j)\in\pUP$ we define  (compare [\cite{andreneto1}, p. 398]).
\begin{equation}\label{ANelsubsets}
\varrho_{i,j} =\begin{cases} \{(i,k)\in\pUP\,|\,i<k<j\,\} & \text{if}\quad (i,j)\in\tril\\
                                  \{(i,k)\in\pUP\,|\,i<k\leq n\,\}\cup\{(\bar j,l)\in\pUP\,|\,\bar j <l\leq n\,\} &   \text{if}\quad (i,j)\in\trir\\
\end{cases}.
\end{equation}

Note that the second case includes $j=\bar i$ for type $\mathfrak{C}_n$, where $\varrho_{i,\bar i} =  \{(i,k)\in\pUP\,|\,i<k\leq n\,\}$.

We set $J_{i,j} = \pUP\setminus \varrho_{i,j}$ and note that $J_{i,j}$ is closed in $\pUP$ as well as $J^\circ_{i,j} = J_{i,j}\setminus (i,j)$. Thus  $J_{i,j}$ arises by removing from $\pUP$ parts of row $i$ respectively of rows $i$ and $\bar j$.  We denote the pattern subgroups $U_{J_{i,j}}$ and $U^\circ_{J_{i,j}}$ by $U_{i,j}$ and $U^\circ_{i,j}$ respectively for short. Moreover it is not hard to check that $U^\circ_{i,j}$ is a normal subgroup of $U_{i,j}$, (see [\cite{andreneto1}, p. 399]).\hfill$\square$
\end{Defn}

Recall that we fixed a non trivial linear character $\theta:(\F_q,+)\rightarrow\C^{*}$. For $\alpha\in\F_q$ define 
$$
\theta_\alpha:(\F_q,+)\rightarrow \C^*:\lambda\mapsto\theta(\alpha\lambda).
$$
Then  $\{\theta_\alpha\,|\,\alpha\in\F_q\,\}$ is the set of all linear characters of the additive group of $\F_q$.

\begin{Lemma}\label{elementarychar}(see [\cite{andreneto1}, p.399]) Let $(i,j)\in\pUP$ and $0\not =\alpha\in\F_q$. Then $\chi^{i,j}_\alpha:U_{i,j}\rightarrow\C^*: u\mapsto \theta_\alpha(u_{ij}) = \theta(\alpha u_{ij})$ defines a nontrivial linear character of $U_{i,j}$. \hfill$\square$
\end{Lemma}

The induced character $\xi^{i,j}_\alpha = \Ind_{U_{i,j}}^U(\chi^{i,j}_\alpha)$ is called  {\bf elementary character} associated with $(i,j)\in\pUP$ and $\alpha\in\F^*_q$ in [\cite{andreneto1}].  From \ref{ANelsubsets} one sees by direct inspection that 
\begin{equation}\label{eldegree}
\deg(\xi^{i,j}_\alpha) = q^{|\varrho_{i,j}|} = [U:U_{i,j}] = 
\begin{cases}q^{j-i-1} & \text{if}\quad (i,j)\in\UP\\
                    q^{n-i} & \text{if} \quad j = \bar i.
\end{cases}
\end{equation}

Let $[A]\in\hat V$ be a core with $\mc(A) = \{(i,j)\}$ and $A_{ij} = \alpha$. Thus by \ref{coredef} $\verge(A) = \alpha e_{ij}$ and in particular $[A]$ is a staircase core. Obviously $A=\alpha e_{ij}$ if $j\leq \tilde n$ or $U$ is of type $\mathfrak{C}_n$ and $A=\alpha e_{ij} + \beta e_{i\bar j}$ for some $\beta \in \F_q$ if $n < j$ and $U$ is of type $\mathfrak{B}_n$ or $\mathfrak{D}_n$. In this case $(i,\bar j)$ is the unique minor condition and there are no supplementary conditions. Using \ref{Orbitdescription} and \ref{defofplace} we can describe the orbit $\cO_A$ by filling the positions in  
\begin{equation}\label{placesElOrb}
\Pl(\cO_A) = 
\begin{cases} \varrho_{i,j} = \{(i,k)\in\pUP\,|\,i<k<j\,\} & \text{if}\quad (i,j)\in\tril\\
                      \{(i,k)\in\pUP\,|\,i<k<j\,\}\setminus\{(i,\bar j)\} & \text{if}\quad (i,j)\in\trir,
\end{cases}
\end{equation}
by arbitrary elements of $\F_q$. Note that all of  these positions are in row $i$. In particular 
 \begin{equation}\label{elOrbSize}
 |\cO_A| = 
 \begin{cases} 
                    q^{j-i-1} = \deg(\xi^{i,j}_\alpha) & \text{if} \quad (i,j)\in\tril\\\
                    q^{j-i-2} = q^{-1}\deg(\xi^{i,j}_\alpha) & \text{for types $\mathfrak{B}_n,\mathfrak{D}_n$, if} \quad (i,j)\in\trir\\
                    q^{j-i-1} = \deg(\xi^{i,j}_\alpha) & \text{for type $\mathfrak{C}_n$ if}\quad (i,\bar i)\not=(i,j)\in\trir\\
                    q^{2(n-i)} = (\deg(\xi^{i,j}_\alpha))^2 & \text{for type $\mathfrak{C}_n$ if} \quad j = \bar i.
                    
\end{cases}
\end{equation}
                    
We write $\cO_A^{i,j} = \{[A].u\,|\,u\in U_{i.j}\,\}$.  Note that $\supp(B)\subseteq \row i$ and $B_{ij}=A_{ij}=\alpha$ for all $[B]\in\cO_A$.  

\begin{Lemma}\label{ijorbits}
Let $[A]\in \hat V$ as above. Then
$$ 
\cO_A^{i,j} = 
\begin{cases}
                    \{[B]\in\cO_A\,|\,B_{ij} = \alpha,  \supp(B)\subseteq \varrho_{i,j}\cup \{(i,j)\}\,\} = \cO_A  & \text{if}\quad (i,j)\in\tril\\
                     \{[B]\in\cO_A\,|\,B_{ik}=0\,\,\text{for}\,\, n<k<j\,\} & \text{if}\quad (i,\bar i) \not = (i,j)\in\trir\\
                     \{[B]\in\cO_A\,|\,B_{i\bar i}=\alpha,\, \supp(B) \subseteq\varrho_{i,j}\cup\{(i,\bar i)\}\,\} &  \text{if} \quad j = \bar i \,\,\text{in case of type $\mathfrak{C}_n $}.
\end{cases}
$$
Therefore in all cases $\supp(B)\cap J_{i,j} = \{(i,j)\}$  and in addition $B_{ij}= \alpha$ for all $[B]\in\cO^{i,j}_A$. In particular 
$$
[B]u =  \chi^{i,j}_\alpha (u)[B].u\quad\text{for all $[B]\in\cO^{i,j}_A$ and all}\, u\in U_{i,j}.
$$
\end{Lemma}
\begin{proof}
Inspecting \ref{fLM} and \ref{fM2} one checks all but the last assertions of the lemma immediately. By \ref{monomial} we have  $[B]u = \theta(\kappa(B,u))[B].u$ and  $\theta(\kappa(B,u)) = \theta(B_{ij}u_{ij}) = \theta(\alpha u_{ij}) = \theta_\alpha(u_{ij}) = \chi^{i,j}_\alpha (u)$ by \ref{traceform}, since  $\{(i,j)\}\subseteq\supp(B)\cap \supp(u) \subseteq \supp(B)\cap J_{i,j} = \{(i,j)\}$.
\end{proof}

Fix $(i,j)\in\pUP$ and $\alpha\in\F_q^*$ and let $f\in\C U_{i,j}$ be the central idempotent such that $\C f$ affords the linear character $\chi = \chi^{i,j}_\alpha$. Then $M = \Ind_{U_{i,j}}^U(\C f)$ affords $\xi = \xi^{i,j}_\alpha$.

\begin{Lemma}\label{ElHoms} Keep the notation introduced above and let
 $$
 x = \sum_{B\in\cO_{i,j}}[B]\in \C\cO^{i,j}_A \subseteq \C\cO_A.
 $$
 Then $\C x$ is a one dimensional submodule of $\RRes^U_{U_{ij}}(\C\cO_A)$ isomorphic to $\C f$. As a consequence there is a nontrivial $U$-homomorphism from $M$ to $\C\cO_A$.
\end{Lemma}
\begin{proof}
Applying lemma  \ref{ijorbits} we obtain for any $u\in U_{i,j}$:
$$
xu =  \sum_{B\in\cO_{i,j}}[B]u =  \chi(u)\sum_{B\in\cO_{i,j}}[B].u =  \chi(u) u
$$
As a consequence, using Frobenius reciprocity, we get:
$$
(0) \not= \Hom_{\C U_{i,j}}(\C f,\RRes^U_{U_{i,j}}(\C\cO_A)) \cong \Hom_{\C U}(M,\C\cO_A),
$$
proving the lemma.
\end{proof}

If $U$ is of type $\mathfrak{B}_n$ or $\mathfrak{D}_n$ and if $(i,j)\in\trir$ and $\beta\in \F_q$ we want to distinguish cores $[A]$ with entry $\beta$ at the minor condition. Thus we write $A_\beta = \beta e_{i\bar j} + \alpha e_{ij}\in V$, if $\tilde n<j<\bar i$ and $A = A_0 = \alpha e_{ij}\in V$ for all types and all $j$. In particular $A_\beta$ is a core with $A_0 = A = \verge(A_\beta)$ for all $\beta\in\F_q$. 

\begin{Lemma}\label{mincond} Suppose $U$ is of type $\mathfrak{B}_n$ or $\mathfrak{D}_n$ and let $(i,j)\in\trir$. Then the distinct orbit modules $\C\cO_{A_\beta}, \beta \in \F_q$ have no irreducible constituent in common and hence afford orthogonal  characters.
\end{Lemma}
\begin{proof} Inspecting \ref{4action} we see that the $X_{i\bar j}$ is contained in the stabilizer $\Stab_U[B]$ for all $[B]\in\cO_{A_\beta}$ and acts on $[B]$ by the linear character $\theta_\beta$ (see \ref{monomial} and \ref{rootelmaction}). Therefore, if $\beta,\gamma\in\F_q$ with $\beta\not=\gamma$, the orbit modules $\C\cO_{A_\beta}$ and $\C\cO_{A_\gamma}$ cannot have an irreducible constituent in common and hence afford orthogonal characters.
\end{proof}

\begin{Defn}\label{vergemodule} For any staircase character $[A]\in\hat V$ we let $\cV (A)$ be the set of characters $[B]\in\hat V$ satisfying $\verge(B)=\verge(A)$ and let $\C\cV(A)$ be the $\C U$-module with monomial basis $\cV (A)$. Thus
$$
\C\cV(A) = \bigoplus_{\verge(\cO_B) = \verge(A)}\C\cO_B
$$
It is an immediate consequence of theorem \ref{monomial} and remark \ref{compTypeA}, that $\cV(A) = \tilde\cO_A = \{[A].u\,|\,u\in\widetilde U\}$ is an Andr\'{e}-Yan orbit yielding a supercharacter of $\widetilde U$. 
\hfill$\square$
\end{Defn}  

We can now describe the elementary Andr\'{e}-Neto characters in terms of our orbit modules:

\begin{Theorem}\label{identificationANorbits} Let $(i,j)\in \pUP$, $0\not=\alpha\in\F_q$ and let $A=\alpha e_{ij}\in V$. As above let $M$ be a $\C U$-module affording the elementary character $\xi^{i,j}_\alpha$. 
\begin{itemize} 
\item[1)] If $(i,j)\in\tril$, or if $U$ is of type $\mathfrak{C}_n$ and $(i,j)\in \UP$, then $M$ is irreducible and $M\cong \C\cO_A$.
\item[2)] If $U$ is of type $\mathfrak{B}_n$ or $\mathfrak{D}_n$ and $(i,j)\in\trir$, then  
$$
M\cong \bigoplus_{\beta\in\F_q} \C\cO_{A_\beta} = \bigoplus_{[B]\in\hat V, \verge[B]=[A]}[B] = \C\cV(A)
$$
is a decomposition of M into a direct sum of $q$ many irreducible $\C U$-modules.  
\item[3)] If $U$ is of type $\mathfrak{C}_n$ and $j=\bar i$, then $M$ is irreducible and occurs with multiplicity $1$ in $\C\cO_A$. 

\end{itemize}
\end{Theorem}

\begin{proof}
By [\cite{andreneto1}, 2.1]  $M$ is irreducible in cases 1) and 3) and decomposes into a sum of $q$ many irreducible pairwise non isomorphic $\C U$-modules in case 2). Now \ref{ElHoms} implies, that $M$ and $\C\cO_{A_\beta}$ have a composition factor in common for all $\beta\in\F_q$. Consequently $M$ is a submodule of $\C\cO_A = \C\cO_{A_0}$ in cases 1) and 3), and hence $M\cong\C\cO_A$ in case 1) by \ref{elOrbSize}. This proves case 1).

If $U$ is of type $\mathfrak{B}_n$ or $\mathfrak{D}_n$ and $(i,j)\in\trir$, then \ref{mincond} implies that the common composition factors $S_\beta$ of $M$ and $\C\cO_{A_\beta}$ for $\beta\in\F_q$  are pairwise non isomorphic and we conclude $M\cong \oplus_{\beta\in\F_q}S_\beta$. Again by \ref{elOrbSize} we find $\C\cO_{A_\beta}\cong S_\beta$ is irreducible and case 2) follows. Finally case 3) follows directly from [\cite{andreneto2},2.6], (indeed there are the other irreducible constituents of $\C\cO_A$ determined as well).
\end{proof}

\begin{Remark}\label{HimLeMag} In [\cite{HLM}] by different methods some irreducible characters of $p$-Sylow subgroups of untwisted finite Chevalley groups are investigated in terms of certain families. These are labeled by single positive roots of the Lie type in question. Each family of single root characters contains a collection of characters of minimal degree, called midafi.  In case of $p$-Sylow subgroups $U$ of classical untwisted type these are afforded  by our orbit modules generated by cores with one main condition corresponding to the single root there. 
\end{Remark}

As an immediate consequence in view of \ref{compTypeA} we obtain:

\begin{Cor}\label{ElCharResTypA} The elementary character $\xi^{i,j}_\alpha$ is  afforded by the restriction to $U$ of the $\widetilde U$-orbit module $\C\cV(A)$ with $A= \alpha e_{ij}$, except if $U$ is of  type $\mathfrak{C}_n$ and $j=\bar i$, where  $\xi^{i,\bar i}_\alpha$ is an irreducible constituent hereof.
\hfill$\square$
\end{Cor}

Next we inspect the notion of basic subsets of positive roots as defined in [\cite{andreneto1}, p.396] in view of our setting. 

\begin{Defn}\label{basicset} A subset $\cD\subseteq\UR=\tilde \Phi^+ = \{(i,j)\,|\,1\leq i<j\leq N\}$ of the set $\UR$ of positions to the north of the diagonal is called {\bf basic}, if it satisfies:
\begin{itemize}
\item[i)] $(i,j)\in\cD$ if and only if $(\bar j, \bar i)\in \cD$, (so $\cD$ is mirror symmetric with respect to the antidiagonal).
\item[ii)] Each row and each column of $\UR$ contains at most one element of $\cD$.  \hfill$\square$
\end{itemize}
 
\end{Defn}

Note that i) ensures, that $\cD\cap\pUP$ already determines $\cD$ completely and hence $\cD$ can be interpreted as set of positive roots of type $\mathfrak{B}_n, \mathfrak{C}_n$ and $\mathfrak{D}_n$ respectively. Obviously if $\cD$ satisfies ii) then $\cD\cap\pUP$ does as well, but condition ii) for all of $\cD$ (in conjunction with condition i)) is stronger: It implies as well, that on the arm  $\cA(i,\bar j) = \{(j,k)\in\pUP\,|\, j<k\leq\bar j\}$ as defined in \ref{lowerhook}, there is no further position contained in $\cD$, if $(i,\bar j)\in\cD\cap\trir$.

\begin{Example}\label{nonmainsep} If $1<j<k<n$ and $M_1=(i,\bar j), M_2=(j,\bar k)$ are contained in $\cD$ then $\bar M_2=(k,\bar j)$ is contained in  $\cD$ as well as $M_1$, contradicting ii):
%\begin{equation*}\label{contra} 
%\includegraphics[width=0.8\textwidth]{./GeoHom/contra}
%\end{equation*}
\end{Example}

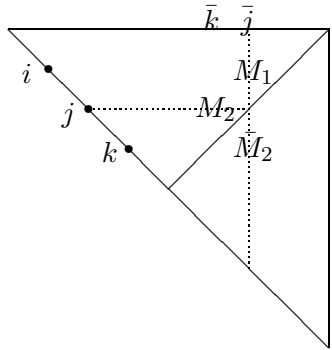
\begin{center}
\begin{picture}(120,120)
\put(0,120){\line(1,0){120}}
\put(120,0){\line(0,1){120}}
\put(0,120){\line(1,-1){120}}
\put(45,75){\circle*{3}} %k
\put(35,70){\makebox{$k$}}
\put(73,120){\makebox{$\bar k$}}
\put(30,90){\circle*{3}} %j
\put(20,85){\makebox{$j$}}
\put(87,120){\makebox{$\bar j$}}
\put(15,105){\circle*{3}} %i
\put(5,100){\makebox{$i$}}
\put(84,101){\makebox{$M_1$}}
\put(70,87){\makebox{$M_2$}}
\put(84,72){\makebox{$\bar M_2$}}

\dottedline{2}(30,90)(90,90)
\dottedline{2}(90,120)(90,30)
\put(60,60){\line(1,1){60}}

\end{picture}
\end{center}

Since $\cD\cap\pUP$ satisfies ii), it can be interpreted as main conditions of characters $[A]\in\hat V$. Then condition ii) for $\cD$ in conjunction with condition i) translates into the requirement for $\main(A)=\cD\cap\pUP$ to be main separated according to the following definition:

\begin{Defn}\label{mainsep}
 The main conditions $\main(A)$ of a character $[A]\in \hat V$  are called {\bf separated} if %$[A]$ is a staircase character and 
${\Limb}(A)\cap \main(A)=\emptyset$. If so, we call $[A]$ and $\cO_A$ {\bf main separated}.\hfill$\square$
\end{Defn}

 Note, that if $[A]$ is main separated, then $[A]$ is staircase as well, since a character $[A]$ which is not staircase, contains at least one column with  two main conditions, such that the lower one is on the leg of the higher one or is on the antidiagonal with the higher one being a right main condition. For staircase characters $[A]\in\hat V$ main separated requires in addition, that on the arms to right main conditions there are no further main conditions. 
 
 \begin{Defn}\label{ANsupchar}(see [\cite{andreneto1}]).
 Given a nonempty basic subset $\cD$ of $\UR$ and a map $\Phi:\cD\cap\pUP\rightarrow \F_q^*$ the Andr\'{e}-Neto supercharacter $\xi_{\cD,\Phi}$ is defined as follows: 
$$
\xi_{\cD,\Phi} = \prod_{(i,j)\in\cD\cap\pUPs}\xi^{i,j}_{\Phi(i,j)},
$$
where the elementary character $\xi^{i,j}_{\Phi(i,j)}$ has been defined above. If $\cD$ is the empty set, the corresponding supercharacter is defined to be the trivial one.
\hfill$\square$ \end{Defn}

We set $B(i,j) = \Phi(i,j)e_{ij}\in V_{\URs}$ for $(i,j)\in\cD$. Moreover we define 
$$
\widetilde\cO_C = \{[C].u\,|\,u\in\widetilde U\}\quad\text{for}\,\, [C] \in \hat V_{\pUPs},
$$ 
identifying the $\C\widetilde U$-modules $\hat V_{\pUPs}$ and $\hat{\widetilde V}_{\pUPs}$ as in \ref{compTypeA}. 

If $U$ is not of type $\mathfrak{C}_n$ or $(i,j)\in\UP$ for all $(i,j)\in\cD\cap\pUP$ we obtain applying \ref{ElCharResTypA},  that  $\xi_{\cD,\Phi}$ is afforded by the restriction to $U$ of the $\C\widetilde U$-module 
\begin{equation}\label{superchartensorproduct}
\bigotimes_{(i,j)\in\cD\cap\pUPs}\C\widetilde{\cO}_{B(i,j)},
\end{equation}
which is 
\begin{equation}\label{vergemoddec}
 \bigotimes_{(i,j)\in\cD\cap\pUPs}\C\cV(B(i,j)).
\end{equation}
 In any case $\xi_{\cD,\Phi}$  is afforded by a direct summand of the tensor product \ref{vergemoddec}. Wee call the $\C\widetilde U$-module of \ref{superchartensorproduct} {\bf Andr\'{e}-Neto module}. 

One checks easily that the condition ii) of \ref{basicset} says that the hypothesis of [\cite{yan}, 6.2] is satisfied. Applying this we have shown that  the 
Andr\'{e}-Neto module \ref{superchartensorproduct} is an Andr\'{e}-Yan  $\widetilde U$-orbit module and we can describe its restriction to $U$:

\begin{Theorem}\label{DecANorbits} Let $\emptyset\not=\cD\subseteq \UR$ be a basic set and let $\Phi:\cD\cap\pUP\rightarrow\F_q^*$ be a map. Set 
$$
A=A(\cD,\Phi) = \sum_{(i,j)\in\cD\cap\pUPs}\Phi(i,j)e_{ij}\in V_{\pUPs}\leq V_{\URs}.
$$
If $j\not=\bar i$ 
for all $(i,j)\in\cD\cap\pUP$, the Andr\'{e}-Neto supercharacter $\xi_{\cD,\Phi}$ is afforded by the the restriction to $U$ of the $\widetilde U$-orbit module $\C\widetilde{\cO}_A$. If $(i,\bar i)\in\cD$ for some $1\leq i\leq n$, then $\xi_{\cD,\Phi}$ is afforded by a direct summand hereof. Moreover, $[A]\in\hat V$   is a main separated verge and $\widetilde{\cO}_A$ decomposes under the action of  $U\leq \widetilde U$ into the disjoint union of all $U$-orbits of the form $\cO_B$ with $[B]\in\hat V$ a main separated core with $\verge(B)=A$. Thus 
$$
\Res ^{\widetilde U}_U (\C\widetilde {O}_A) = \C\cV(A) = \bigoplus_{\text{cores\,}[B]\in\hat V\atop \verge(B) = A} \C\cO_B
$$
is a direct sum decomposition of the Andr\'{e}-Neto module attached to the supercharacter $\xi_{\cD,\Phi}$  into $U$-orbit modules. \hfill$\square$
\end{Theorem}

 Rewriting tensor products in \ref{superchartensorproduct} using  [\cite{yan}, 6.2] into $\widetilde U$-orbit modules works for any  $\cD\subseteq\UR$ satisfying condition ii) of \ref{basicset}. The requirement that  $\cD$ satisfies  in addition condition i) of \ref{basicset} is needed to show, that the Andr\'{e}-Neto supercharacters of $U$ are pairwise orthogonal.
 
 In [\cite{GJD2}] we determine the stabilizers in $U$ of main separated cores and prove, that every irreducible $\C U$-module is constituent of some main separated orbit module. Moreover we show, that the $U$-orbit modules $\cO_A$ for main separated cores $[A]\in\hat V$ with $\supp(A)\subseteq \UP$ are either isomorphic or afford orthogonal characters. In type $\mathfrak{C}_n$ if $\supp(A)\not\subseteq\UP$, that is, if $(i,\bar i)\in\main(A)$ for some $1\leq i< n$ this is not true in view of part 3) of theorem \ref{identificationANorbits}. This case requires a modification of the original setting and is presently under investigation.

 %\section*{Acknowledgement}
 %We would like to thank the referee for interesting comments and valuable suggestions.

\providecommand{\bysame}{\leavevmode ---\ }
\providecommand{\og}{``} \providecommand{\fg}{''}
\providecommand{\smfandname}{and}
\providecommand{\smfedsname}{\'eds.}
\providecommand{\smfedname}{\'ed.}
\providecommand{\smfmastersthesisname}{M\'emoire}
\providecommand{\smfphdthesisname}{Th\`ese}

\end{document}